\documentclass[10pt]{amsart}

\usepackage[T1]{fontenc}  
\usepackage[utf8]{inputenc} 
\usepackage{longtable}
\usepackage{amssymb}
\usepackage{amsthm}
\usepackage{amsmath}
\usepackage{wasysym}
\usepackage{dcpic, pictex}
\usepackage{bbm}
\usepackage{enumerate}
\usepackage{amsfonts}
\usepackage{amsthm}
\usepackage{amsmath}
\usepackage{times}
\usepackage{amscd}
\usepackage[mathscr]{eucal}
\usepackage{indentfirst}
\usepackage{pict2e}
\usepackage{epic}
\usepackage{epstopdf} 
\usepackage{verbatim}
\usepackage{cancel}
\usepackage[foot]{amsaddr}
\usepackage{lipsum}
\usepackage{mathrsfs}
\usepackage{bm}

\usepackage{subcaption}
\usepackage{graphicx}

\usepackage{
amsfonts, 
latexsym, 
enumerate,
}
\usepackage[english]{babel}
\allowdisplaybreaks
\makeindex
\newtheorem{theorem}{Theorem}[section]
\newtheorem{lemma}[theorem]{Lemma}
\newtheorem{proposition}[theorem]{Proposition}
\newtheorem{definition}[theorem]{Definition}

\newtheorem{remark}[theorem]{Remark}
\newtheorem{corollary}[theorem]{Corollary}

\usepackage[a4paper,top=3.5cm,bottom=3.5cm,left=3cm,right=3cm]{geometry}
\numberwithin{equation}{section}

\allowdisplaybreaks

\usepackage{ bbold }
\usepackage{xcolor}

\newcommand{\nc}{\normalcolor}
\newcommand{\dif}{\mathrm{d}}
\newcommand{\E}{\mathbf{E}}
\newcommand{\Prob}{\mathbf{P}}
\newcommand{\R}{\mathbf{R}}
\newcommand{\C}{\mathbf{C}}

\newcommand{\N}{\mathbf{N}}

\newcommand{\X}{\mathcal{X}}

\newcommand{\ii}{\mathrm{i}}
\newcommand{\ee}{\mathrm{e}}

\usepackage{hyperref}

\title[Law of fractional logarithm for random matrices]{Law of fractional logarithm for random matrices}

\date{\today}

\begin{document}

\maketitle

\vspace{0.25cm}

\renewcommand{\thefootnote}{\fnsymbol{footnote}}

\noindent
\mbox{}%
\hfill%
\begin{minipage}{0.19\textwidth}
	\centering
	{Zhigang Bao}\footnotemark[1]\\
	\footnotesize{\textit{zgbao@hku.hk}}
\end{minipage}
\hfill%
\begin{minipage}{0.19\textwidth}
	\centering
	{Giorgio Cipolloni}\footnotemark[2]\\
	\footnotesize{\textit{gcipolloni@arizona.edu}}
\end{minipage}
\hfill%
\begin{minipage}{0.19\textwidth}
	\centering
	{L\'aszl\'o Erd\H{o}s}\footnotemark[3]\\
	\footnotesize{\textit{lerdos@ist.ac.at}}
\end{minipage}
\hfill%
\begin{minipage}{0.19\textwidth}
	\centering
	{Joscha Henheik}\footnotemark[3]\\
	\footnotesize{\textit{joscha.henheik@ist.ac.at}}
\end{minipage}
\hfill%
\begin{minipage}{0.19\textwidth}
	\centering
	{Oleksii Kolupaiev}\footnotemark[3]\\
	\footnotesize{\textit{okolupaiev@ist.ac.at}}
\end{minipage}
\hfill%
\mbox{}%
\footnotetext[1]{Department of Mathematics, The University of Hong Kong.}
\footnotetext[1]{Supported by Hong Kong RGC Grant GRF 16304724, NSFC12222121 and NSFC12271475.}
\footnotetext[2]{Department of Mathematics, University of Arizona, 617 N Santa Rita Ave, Tucson, AZ 85721, USA.}
\footnotetext[3]{Institute of Science and Technology Austria, Am Campus 1, 3400 Klosterneuburg, Austria. 
}
\footnotetext[3]{Supported by the ERC Advanced Grant ``RMTBeyond'' No.~101020331.}

\renewcommand*{\thefootnote}{\arabic{footnote}}
\vspace{0.25cm}

\begin{abstract} We prove the Paquette-Zeitouni law of fractional logarithm (LFL) for the extreme eigenvalues  \cite{PZ_LIL}
 in full generality, and thereby verify a conjecture from \cite{PZ_LIL}. Our result holds for  any Wigner minor process  and
both symmetry classes, in particular for the GOE minor process, while \cite{PZ_LIL} and the recent full resolution of LFL 
 by Baslingker et.~al.~\cite{baslingker2024paquette} cover only  the GUE case which is determinantal.
 Lacking the possibility for a direct comparison with the Gaussian case, we develop a robust and natural  method
 for both key parts of the proof. On one hand, 
 we rely on a powerful martingale technique 
 to describe precisely  the strong correlation  between the largest eigenvalue of an $N\times N$ Wigner matrix
 and its $(N-k)\times (N-k)$ minor if $k\ll N^{2/3}$. On the other hand, we use dynamical methods 
 to show that this correlation is weak if $k\gg N^{2/3}$. 
\end{abstract}
\vspace{0.15cm}

\footnotesize \textit{Keywords:} Law of iterated logarithm, Local Law, Zigzag Strategy, Dyson Brownian motion, decorrelation estimate. 

\footnotesize \textit{2020 Mathematics Subject Classification:} 60B20, 60G55, 82C10.
\vspace{0.25cm}
\normalsize

\section{Introduction and main results}

Consider the largest eigenvalue $\widetilde\lambda_1^{(N)}$ of an $N\times N$ standard Wigner matrix, i.e. a
random Hermitian matrix $H$ with i.i.d. entries with the standard normalisation $\E h_{ij}=0$,
$\E |h_{ij}|^2=1/N$. It is well known that $\widetilde\lambda_1^{(N)}\to 2$ almost surely as $N\to\infty$
and the fluctuation of its appropriate rescaling, $\lambda_1^{(N)}: = N^{2/3}(\widetilde\lambda_1^{(N)}- 2)$,
follows the universal Tracy-Widom law  \cite{TW_level-spacing, TW_Orth_ME}. These fundamental results can be viewed as analogues of
the law of large numbers and the central limit theorem for $S_N := N^{-1/2}\sum_{j=1}^N Y_j$, i.e. for a rescaled
sums of i.i.d. random variables $Y_j$ with zero mean and unit variance. It is also well known 
that the extremal fluctuations
of $S_N$ are governed by the Hartman-Wintner \cite{hartman1941law} 
{\it law of iterated logarithm (LIL)} asserting that
\begin{equation}\label{LIL}
  \liminf_{N \to \infty} \frac{S_N}{\sqrt{2\log \log N}}=-1, \quad  \limsup_{N \to \infty} \frac{S_N}{\sqrt{2\log \log N}}=1,\quad \mbox{almost surely.}
\end{equation}

Motivated by a very natural question of Gil Kalai \cite{Kalai},
 Paquette and Zeitouni in \cite{PZ_LIL} considered  an analogue of \eqref{LIL}
for the extremal fluctuations of $\lambda_1^{(N)}$.  Note that, unlike the Tracy-Widom limit,  this question concerns
the joint distribution of the entire sequence $\lambda_1^{(N)}$. Therefore, the sequence $H^{(N)}$ needs to be
defined consistently by taking appropriately scaled $N\times N$ upper left corners of 
a doubly infinite  symmetric \nc array of i.i.d. random variables (the precise definition of this {\it Wigner minor process} will be given
later). Paquette and Zeitouni discovered that for complex Gaussian matrices (GUE ensemble)
a complete analogue of LIL holds in the following form
\begin{equation}\label{LFL1}
\liminf_{N \to \infty}\frac{\lambda_1^{(N)}}{(\log N)^{1/3}} = -4^{1/3}, \quad \limsup_{N \to \infty}\frac{\lambda_1^{(N)}}{(\log N)^{2/3}} = 
\Big(\frac{1}{4}\Big)^{2/3}, \quad \mbox{almost surely},
\end{equation}
and they coined it the {\it law of fractional logarithm (LFL)}. In fact,  they proved  the  $\limsup$ part
of \eqref{LFL1},
while for $\liminf$ the precise constant was  conjectured and only upper and lower limits were proven in \cite{PZ_LIL}.
The conjectured constant was proven only  very recently by Baslingker et.~al.~\cite{baslingker2024paquette}.

Our main result is the proof of LFL \eqref{LFL1} for any Wigner minor process without Gaussianity assumption
on the matrix elements.  Moreover, we can handle both real symmetric $(\beta=1)$ and complex
Hermitian $(\beta=2)$ symmetry classes (the constants in \eqref{LFL1} depend on $\beta$, see \eqref{eq:LFL}), 
 while the proofs in \cite{PZ_LIL} and \cite{baslingker2024paquette} were restricted to the determinantal GUE minor process. This solves the natural LFL universality conjecture of Paquette and Zeitouni in~\cite{PZ_LIL} for both symmetry classes.
 Our method is very robust, in fact it can be extended well beyond the i.i.d. case.
Finally, we believe that our proof is also very natural since,
for its  most critical part, it mimics Kolmogorov’s classical  proof 
\cite{kolmogorov1931analytischen}
of LIL and its extension to martingales  \cite{stout1970martingale}.
The corresponding steps  in \cite{PZ_LIL} and \cite{baslingker2024paquette} are technically much heavier and 
rely on specific determinantal formulas available only for GUE. 

\medskip

In order to explain the new insights in our  proof, we recall 
 the basic heuristic mechanisms behind \eqref{LIL} and \eqref{LFL1}.
We start with noticing  two main differences between \eqref{LIL} and \eqref{LFL1}.
First, unlike in \eqref{LIL}\nc, the $\limsup$ and $\liminf$ results in \eqref{LFL1} \nc are genuinely different, corresponding to the fact that
the Tracy-Widom distribution is strongly asymmetric, in particular its left tail  decays much faster, as $\exp{(-x^3/12)}$,
than  its right tail with an $\exp{(-4x^{3/2}/3)}$ decay. Second, the  double logarithmic factor 
 in \eqref{LIL} is changed to a fractional power of the logarithm in \eqref{LFL1}.
Recall that the $\sqrt{2\log\log N}$ factor in \eqref{LIL} is the consequence of two effects: the correlation structure of
the sequence $S_N$ and the $\exp{(-x^2/2)}$ tail behavior of the standard Gaussian limit of $S_N$.
Indeed, $S_N$ and $S_{N+M}$ are strongly correlated
if $M\ll N$ and become essentially independent if $M\gg N$. Roughly speaking among the random variables
$S_1, S_2, \ldots, S_N$ only $\log N$ are (essentially) independent and the tail probability of $S_N$
 is exactly $1/\log N$
at level  $x=\sqrt{2\log\log N}$. 
Back to \eqref{LFL1}, note that both fractional powers are chosen such that the tail probability in the  corresponding 
moderate deviation regime is of order $N^{-1/3}$. This choice is justified by the remarkable fact that 
the threshold for (in)dependence between $\lambda_1^{(N)}$ and $\lambda_1^{(N+M)}$ 
lies at $M\sim N^{2/3}$, thus among $\lambda_1^{(1)}, \lambda_1^{(2)}, \ldots, \lambda_1^{(N)}$
there are roughly $N^{1/3}$ (almost) independent variables. This decorrelation transition at $M\sim N^{2/3}$
for the largest eigenvalues of the minor sequence was first discovered by Forrester and Nagao 
for the GUE case \cite{determinantal_corr} and we recently proved it  in full generality for both symmetry classes
and for general Wigner  matrices \cite{minor}.

Given these heuristics and similar reasonings to \cite{PZ_LIL} and \cite{baslingker2024paquette}, our proof
 has two main inputs. First, we need very accurate lower and upper 
bounds on both the right and left tails of $\lambda_1^{(N)}$
in the moderate deviation regime. This problem has a long history, summarized 
after Proposition~\ref{prop:tails} below. Here  we rely on the very  recent results of Baslingker~et.~al.~\cite{baslingker2024optimal} that proved exactly what we need for GOE and GUE ensembles.
 These estimates need to be extended to our general Wigner setup, but this is an easy adaptation 
 of the {\it Green function comparison theorem (GFT)} developed in \cite{small_dev}
 for tail estimates.

Our main novelty compared with  \cite{PZ_LIL} and \cite{baslingker2024paquette} 
lies in verifying the  second input, the sharp decorrelation threshold
in the sequence $\lambda_1^{(N)}$.  This input itself has two parts: (i) the {\it decorrelation estimate} proves that $\lambda_1^{(N)}$ and $\lambda_1^{(N+M)}$
are essentially independent if $M\gg N^{2/3}$, while (ii) the
 {\it correlation estimate}
shows that $\lambda_1^{(N)}$ and $\lambda_1^{(N+M)}$ are strongly correlated if $M\ll N^{2/3}$,
i.e. they are much closer to each other
than their natural fluctuation scale.
Both estimates need to be effective in the moderate deviation regimes, while 
our recent work \cite{minor} on this threshold concerned the typical regime. For part (i)
we strengthen our decorrelation proof from \cite{minor} based upon Dyson Brownian Motion (DBM)
and GFT techniques to cover this extended regime as well. The brevity of this actual  proof  in Section~\ref{app:decor} 
is misleading since it heavily relies on DBM techniques from \cite{landon2017edge, Bou_extreme}
combined with eigenvector overlap bounds obtained from multi-resolvent local laws from \cite{eigenv_decorr}
in a way that was introduced in \cite{macroCLT_complex} and developed in several follow-up papers (see e.g. \cite{minor, bourgade2024fluctuations, cipolloni2023quenched, spectral_radius}).
In contrast to this heavy machinery, our correlation proof is surprisingly elementary.
Using a simple Schur complement formula, we can bound  $\lambda_1^{(N+M)}- \lambda_1^{(N)}$ 
from below by a martingale, see \eqref{mart} below. In the actual proof
for the upper bound on $\limsup$ in \eqref{LFL1} (as well as the analogous lower bound
on $\liminf$), we choose an appropriate subsequence $N=N_k\sim k^{3-\epsilon}$
as in \cite[Section 3]{PZ_LIL}. The tail bound together with a Borel-Cantelli
lemma guarantee that along this subsequence $\lambda_1^{(N_k)}$ is not too big almost surely.
The standard  maximal inequality for martingales then easily extends this
fact  to the entire sequence $\lambda_1^{(N)}$. The analogous part of the proof 
in \cite{PZ_LIL} relied on further dyadic decompositions and explicit formulas available only in the GUE setup.
Establishing the decorrelation threshold in \cite{baslingker2024paquette} is even more involved 
as it relies on
 Baryshnikov’s distributional identity \cite{baryshnikov2001gues} between 
the largest eigenvalues of the GUE minor process and the vector of passage times 
in {\it Brownian last passage percolation}.
We believe that our martingale approach handles the core of this problem in the most natural way,
and, more importantly, it directly applies to general Wigner matrices.

\nc

\subsection{Main results}
Our main object of interest is the \emph{Wigner minor process}: 
\begin{definition}[Wigner minor process]  \label{def:minor}
Let $X = (x_{ij})_{i,j \in \N}$ be a doubly infinite array of independent (up to  
symmetry $x_{ij} = \overline{x}_{ji}$) real ($\beta = 1$) or complex ($\beta = 2$) random variables. We assume them to be centered, $\E x_{ij} = 0$ with variances given by 
\begin{equation} \label{eq:Xdef}
\E x_{ij}^2=1+\delta_{ij} \quad \text{for} \quad \beta = 1 \qquad \text{and} \qquad \E |x_{ij}|^2 = 1 \,, \quad \E x_{ij}^2 = \delta_{ij} \quad \text{for} \quad \beta = 2 \,. 
\end{equation}
We assume that all moments of the entries are finite, meaning that there exist constants $C_p > 0$ such that 
\begin{equation} \label{eq:Xdefmoments}
\sup_{i,j \in \N} \E |x_{ij}|^p \le C_p \,. 
\end{equation}
For any $N\in\N$ we introduce 
\begin{equation} \label{eq:Hdef}
	H^{(N)}:=N^{-1/2}X^{(N)},
\end{equation}
where $X^{(N)}$ is the $N \times N$ upper left corner of $X$. Note that  $H^{(N)}$ is a standard real symmetric or complex Hermitian Wigner matrix
for each $N$. The sequence $\big(H^{(N)}\big)_{N \in \N}$
is called the Wigner minor process.
\end{definition}
The eigenvalues of $H^{(N)}$, in decreasing order, will be denoted by
 $\{\widetilde{\lambda}_j^{(N)}\}_{j=1}^{N}$. 
It is well known (see, e.g., \cite{soshnikov1999universality, EYY12, tao2010random}) that, as $N \to \infty$, it holds
\begin{equation}
	\lambda_1^{(N)}:= N^{2/3}(\widetilde{\lambda}_1^{(N)}-2)\Rightarrow \mathrm{TW}_\beta, \qquad \beta=1, 2\,,  \label{eq:TWconv}
\end{equation}
in the sense of distributions, where $\mathrm{TW}_\beta$ denotes the Tracy-Widom distribution with parameter $\beta$ \cite{TW_level-spacing, TW_Orth_ME}. 

In this paper, we show that $\lambda_1^{(N)}$ obeys a universal \emph{law of fractional logarithm}, independent of the precise single entry distributions of $X$. 
We formulate our results for the largest eigenvalue, but they naturally hold for 
the lowest eigenvalue $\lambda_N^{(N)}$ as well.

\begin{theorem}[Law of fractional logarithm] \label{thm:main}
Let $\big(\lambda_1^{(N)}\big)_{N \in \N}$ be the (shifted and scaled according to \eqref{eq:TWconv}) largest eigenvalues of a real symmetric ($\beta = 1$) or complex Hermitian ($\beta = 2$) Wigner minor process $\big(H^{(N)}\big)_{N \in \N}$ as in Definition \ref{def:minor}.  Then, almost surely, we have
\begin{equation} \label{eq:LFL}
	\liminf_{N \to \infty} \frac{\lambda_1^{(N)}}{(\log N)^{1/3}} = - \left(\frac{8}{\beta}\right)^{1/3} \qquad \text{and} \qquad 
\limsup_{N \to \infty} \frac{\lambda_1^{(N)}}{(\log N)^{2/3}} = \left(\frac{1}{2 \beta}\right)^{2/3} \,. 
\end{equation}
\end{theorem}
For the GUE case ($\beta=2$,  $x_{ij}$ are standard Gaussian) the result on $\limsup$ and an explicit 
range on $\liminf$ have been established in the pioneering paper of Paquette and Zeitouni \cite{PZ_LIL}, while 
the precise constant for $\liminf$ was obtained very  recently 
by Baslingker et.~al.~\cite{baslingker2024paquette}. The limiting constants in \eqref{eq:LFL} are
directly related to the exponents in the tail probability estimates for the Tracy-Widom distribution, 
see  Proposition~\ref{prop:tails} and explanation thereafter.

As a corollary to Theorem \ref{thm:main}, we identify all possible limit points. The proof is given in Section \ref{sec:corollary}. 

\begin{corollary}[All possible limit points] \label{cor:limits}  Under the assumptions and notations of Theorem \ref{thm:main}, we have
	\begin{align}
		\bigcap_{m=1}^{\infty} \overline{\Big\{ \lambda_{1}^{(N)}/ (\log N)^{2/3} : N\geq m\Big\}}= \left[0, \Big(\frac{1}{2\beta}\Big)^{2/3}\right] \label{closure-2/3}
	\end{align}
	and 
	\begin{align}
		\bigcap_{m=1}^{\infty} \overline{\Big\{\lambda_1^{(N)}/ (\log N)^{1/3}:N\geq m\Big\}}=\left[-\Big(\frac{8}{\beta}\Big)^{1/3}, \infty\right],  \label{closure-1/3}
	\end{align}
where $\overline{E}$ denotes the closure of a set $E\subset \R$. 
\end{corollary}

\begin{remark}[Extensions beyond Wigner matrices]\label{rm:ext} 
 We prove our results for the setup where $x_{ij}$ are independent, 
centred and have unit variance \eqref{eq:Xdef}, i.e.~$H^{(N)}$ is a standard Wigner matrix. However, our methods can 
 handle much more general random matrix ensembles that involve non-centred and correlated entries  
 under the conditions of \cite{slow_corr, cusp_univ}. 
 
 To keep the presentation short, we explain only the 
 uncorrelated situation, i.e. the case of  Wigner-type matrices with diagonal deformation.
This means to consider a deterministic doubly infinite array $\Sigma = (\Sigma_{ij})_{i,j\in \N}$ and an 
infinite vector ${\bm a} = (a_i)_{i\in \N}$ 
such that $c\le \Sigma_{ij}\le C$ and $|a_{i}|\le C$ for any $i,j$ with some fixed positive constants $c, C$. Then, as an extension of
\eqref{eq:Hdef}, we consider
\begin{equation} \label{eq:Hdefgen}
	H^{(N)}:=N^{-1/2}(\Sigma\odot X)^{(N)} + \mathrm{diag} ({\bm a}^{(N)}),
\end{equation}
where $\odot$ indicates the entry-wise matrix multiplication (Kronecker product)
and $\mathrm{diag} ({\bm a}^{(N)})$ is the diagonal matrix with $(a_1, a_2, \ldots, a_N)$ 
in the diagonal. For any fixed $N$, $H^{(N)}$ is a Wigner-type matrix with {\it flatness} condition
\cite{univ_W-type}. 
For large $N$, the density of eigenvalues is close to a deterministic function, 
called the {\it self-consistent density of states (scDOS)} but in general it is not the semicircle law.
Typically the scDOS is supported on finitely many intervals (called bands) with square root singularities at the edges
\cite{sing_QVE, QVE}.
This latter property ensures that the extreme eigenvalues near the edges still obey the Tracy-Widom
law and the number of eigenvalues within each band is constant with very high probability \cite{Alt_band}.

Our methods in the current paper are sufficiently robust to prove Theorem~\ref{thm:main} for
the (appropriately rescaled) smallest and  largest  eigenvalues of $H^{(N)}$ in \eqref{eq:Hdefgen} and even 
for all extreme eigenvalues at the internal band edges assuming that $\Sigma$ and ${\bm a}$ are such that 
the gaps between neighboring bands do not close as $N$ increases.
 This is the case, for example, if ${\bm a}=(1, -1, 1, -1, \ldots)$
and $\Sigma_{ij} := c<1$ for all $i,j$. To keep the current presentation simple, we do not give a full proof
of this generalization, but in Section~\ref{sec:wignertype} we will explain the main ideas.
\end{remark}

\subsection{Notations} 
For positive quantities $f,g$ we write $f\lesssim g$ (or $f=\mathcal{O}(g)$) and $f\sim g$ if $f \le C g$ or $c g\le f\le Cg$, respectively, for some constants $c,C>0$ which only depend on the constants appearing in the moment condition \eqref{eq:Xdefmoments}. 

For any natural number $n$ we set $[n]: =\{ 1, 2,\ldots ,n\}$. Matrix entries are indexed by lowercase Roman letters $a, b, c, ...$ from the beginning of the alphabet. We denote vectors by bold-faced lowercase Roman letters ${\bm x}, {\bm y}\in\C ^N$ for some $N\in\N$. Vector and matrix norms, $\lVert {\bm x}\rVert$ and $\lVert A\rVert$, indicate the usual Euclidean norm and the corresponding induced matrix norm. For any $N\times N$ matrix $A$ we use the notation $\langle A\rangle:= N^{-1}\mathrm{Tr}  A$ for its normalized trace. Moreover, for vectors ${\bm x}, {\bm y}\in\C^N$ we denote their scalar product by $\langle {\bm x},{\bm y}\rangle:= \sum_{i} \overline{x}_i y_i$. 

Finally, we use the concept of ``with very high probability'' \emph{(w.v.h.p.)} meaning that for any fixed $C>0$, the probability of an $N$-dependent event is bigger than $1-N^{-C}$ for $N\ge N_0(C)$. Moreover, even if not explicitly stated, every estimate involving $N$ is understood to hold for $N$ being sufficiently large. \nc  We also introduce the notion of \emph{stochastic domination} (see e.g.~\cite{loc_sc_gen}): given two families of non-negative random variables
\[
X=\left(X^{(N)}(u) : N\in\N, u\in U^{(N)} \right) \quad \mathrm{and}\quad Y=\left(Y^{(N)}(u) : N\in\N, u\in U^{(N)} \right)
\] 
indexed by $N$ (and possibly some parameter $u$  in some parameter space $U^{(N)}$), 
we say that $X$ is stochastically dominated by $Y$, if for all $\xi, C>0$ we have 
\begin{equation}
	\label{stochdom}
	\sup_{u\in U^{(N)}} \mathbf{P}\left[X^{(N)}(u)>N^\xi  Y^{(N)}(u)\right]\le N^{-C}
\end{equation}
for large enough $N\ge N_0(\xi,C)$. In this case we use the notation $X\prec Y$ or $X= \mathcal{O}_\prec(Y)$.

\section{Proof of Theorem \ref{thm:main}}\label{sec:2}

In this section, we provide the proof of Theorem \ref{thm:main}. The argument is split into two parts: First, in Section \ref{subsec:decor} we establish a lower bound for $\limsup$ and an upper bound for $\liminf$ via a suitable {\it decorrelation estimate} on the moderate large deviation tail events for  $\lambda_1^{(N)}$ for different $N$’s. 
 Afterwards, in Section \ref{subsec:maxmart}, we obtain the complementary upper bound for $\limsup$ and lower bound for $\liminf$ via martingale estimates. Both
 parts of the proof rely on the following small deviation tail estimates for the largest eigenvalue of Wigner matrix. 
\begin{proposition}[Small deviation tail estimates] \label{prop:tails}
Fix any $K > 0$. Let $\lambda_1^{(N)}$ be the (shifted and scaled according to \eqref{eq:TWconv}) largest eigenvalue of a real symmetric ($\beta = 1$) or complex Hermitian ($\beta = 2$) Wigner matrix as in \eqref{eq:Xdef}--\eqref{eq:Hdef}. Then, we have the following: 
\begin{itemize}
\item[(a)] \textnormal{[Right tail]} For any $\epsilon > 0$ there exists a constant $C = C(K, \epsilon)$ such that for any $1 \le x \le K (\log N)^{2/3}$ we have
\begin{equation} \label{eq:righttail}
C^{-1} \ee^{- \frac{2 \beta}{3}(1+\epsilon) x^{3/2}} \le \Prob \left[\lambda_1^{(N)} \ge x\right] \le C \ee^{- \frac{2 \beta}{3}(1-\epsilon) x^{3/2}}
\end{equation} 
for any sufficiently large $N \ge N_0(K, \epsilon)$. 
\item[(b)] \textnormal{[Left tail]}  For  any $\epsilon > 0$ there exists a constant $C = C(K, \epsilon)$ such that for any $1 \le x \le K (\log N)^{1/3}$ we have
\begin{equation} \label{eq:lefttail}
C^{-1} \ee^{- \frac{ \beta}{24} (1 + \epsilon)x^{3}} \le 	\Prob \left[\lambda_1^{(N)} \le  - x\right] \le C \ee^{- \frac{ \beta}{24} (1 - \epsilon)x^{3}}
\end{equation} 
for all sufficiently large $N \ge N_0(K, \epsilon)$. 
\end{itemize}
\end{proposition}

The precise exponents in \eqref{eq:righttail}--\eqref{eq:lefttail} coincide with the large deviation tail estimates
for the Tracy-Widom distribution proven by Ramirez, Rider, and Virág in \cite{ramirez2011beta} (following
a heuristics  in \cite{edelman2007random}).
However, the  distributional limit of $\lambda_1^{(N)}$ to  the Tracy-Widom law does not 
imply that the same tail bounds hold before the limit and for a large ($N$-dependent) range of $x$; these
facts needed separate proofs.

Historically, the right-tail estimate \eqref{eq:righttail} was firstly established for the GUE case in \cite[Lemma 7.3]{PZ_LIL}. This was followed by an upper bound on the right tail for the GOE case \eqref{eq:righttail} in \cite[Lemma 1.1]{small_dev}. In \cite[Theorem 1.2]{small_dev}, the upper bounds on both right and left tails were further extended to the Wigner case for both symmetry classes. Finally, \eqref{eq:righttail} and \eqref{eq:lefttail} were proven both for GOE and GUE cases in full generality in \cite{baslingker2024optimal}. These developments were preceded by a series of intermediate results; see, e.g., \cite{Ledoux07} for a review. Now we explain how to establish both \eqref{eq:righttail} and \eqref{eq:lefttail} for general Wigner matrices.

\begin{proof}[Proof of Proposition \ref{prop:tails}] Our approach relies on the Green function comparison method from \cite{small_dev} to compare the Wigner case with the Gaussian one from \cite{baslingker2024optimal}. In \cite{small_dev} the upper bounds on the tails in \eqref{eq:righttail} and a non-optimal version of \eqref{eq:lefttail} were used in the Gaussian case as an input, and then extended to the Wigner case by tracking events of polynomially small probability during the comparison process. Applying the same method with a complete set of bounds from \cite{baslingker2024optimal} as an input allows us to establish \eqref{eq:righttail} and \eqref{eq:lefttail} for Wigner matrices.
\end{proof}

\subsection{Decorrelation estimate: Lower bound for $\limsup$ and upper bound for $\liminf$} \label{subsec:decor}

In order to formulate our decorrelation estimate, for $N_1, N_2 \in \N$ and $x_1, x_2 \ge 0$, we introduce the tail events
\begin{equation} \label{eq:tailevents}
\mathfrak{F}_l(x_l) := \left\{ \lambda_{1}^{(N_l)} \ge x_l (\log N_l)^{2/3}\right\} \quad \text{and} \quad \mathfrak{E}_l(x_l) := \left\{ \lambda_{1}^{(N_l)} \le  - x_l (\log N_l)^{1/3}\right\} \quad \text{for} \ l = 1,2\,. 
\end{equation}
Then we have the following result, whose proof is given in Section \ref{app:decor}. 
\begin{proposition}[Decorrelation estimate on tail events] \label{prop:decor}
Fix $\epsilon_0 >0$ and take two large positive integers $N_1,  N_2 \in \N$ such that $N_2^{2/3 + \epsilon_0} \le N_2 -N_1 \le N_2^{1 - \epsilon_0}$. Fix a large constant $K > 0$. Then there exists a constant $\delta >0$ such that 
\begin{subequations} \label{eq:decor}
\begin{equation} \label{eq:rightdecor}
	\Prob \left[\mathfrak{F}_1(x_1) \cap \mathfrak{F}_2(x_2)\right] = \Prob\left[\mathfrak{F}_1(x_1)\right] \, \Prob\left[\mathfrak{F}_2(x_2)\right] \big(1 + \mathcal{O}(N_2^{-\delta})\big) 
\end{equation}
and 
\begin{equation} \label{eq:leftdecor}
		\Prob \left[\mathfrak{E}_1(x_1) \cap \mathfrak{E}_2(x_2)\right] = \Prob\left[\mathfrak{E}_1(x_1)\right] \, \Prob\left[\mathfrak{E}_2(x_2)\right] \big(1 + \mathcal{O}(N_2^{-\delta})\big) 
		\end{equation}
uniformly for $K^{-1} \le x_l \le K $, $l=1,2$. 
\end{subequations}
\end{proposition}

Armed with Proposition \ref{prop:decor}, we now turn to proving a lower bound for $\limsup$ and an upper bound for $\liminf$ in Sections \ref{subsubsec:lowlimsup} and \ref{subsubsec:upplimsup}, respectively.

\subsubsection{Lower bound for $\limsup$} \label{subsubsec:lowlimsup}In this section, we prove that almost surely
\begin{align}
	\limsup_{N\to \infty} \frac{\lambda_1^{(N)}}{(\log N)^{2/3}}\geq c_*, \qquad \text{with} \qquad  c_*:=\left(\frac{1}{2\beta}\right)^{2/3}.  \label{031737}
\end{align}
That means we shall show that for any $c<c^*$, almost surely, $\lambda_1^{(N)}\geq c(\log N)^{2/3}$ occurs infinitely many times. It suffices to find a subsequence so that it is true. To this end, we fix an $\alpha>3$ to be determined below, and choose the 
subsequence\footnote{With a slight abuse of notations, $N_1$ and $N_2$ used in Proposition~\ref{prop:decor}
are not exactly the same as $N_k$ defined here, in fact we will use this proposition in~\eqref{varest}
for pairs $N_i, N_j$ for some $i,j$.}
\begin{align}
	N_k=\lceil k^\alpha\rceil. \label{031740}
\end{align}
For  a fixed $0 < c<c_*$, we define (analogously to \eqref{eq:tailevents}) the event 
\begin{align*}
\mathfrak{F}_k = 	\mathfrak{F}_k(c):=\Big\{ \lambda_1^{(N_k)}\geq c(\log N_k)^{2/3} \Big\}
\end{align*} 
and, for $M\in \N$, the partial sum
\begin{align*}
	S_M:=\sum_{k=1}^M\mathbf{1}(\mathfrak{F}_k). 
\end{align*}
To prove \eqref{031737}, it suffices to show that we can choose $\alpha=\alpha(c)>3$, such that $S_M\to \infty$ almost surely as $M\to \infty$. Since $S_M$ is nondecreasing in $M$, it is enough to show that $S_M\to \infty$ in probability.  
Let $\delta=\delta(c)>0$ be a sufficiently small constant. We define 
\begin{align} \label{eq:SMdelta}
	S_{M,\delta}:=\sum_{k=M-M^{1-\delta}}^M \mathbf{1}(\mathfrak{F}_k)\leq S_M.
\end{align}
Then it suffices to show $ S_{M,\delta}\to \infty$ in probability as $M\to \infty$. 
To this end, it is enough to prove 
\begin{align}
	\text{(i)}: \E S_{M,\delta}\to \infty, \qquad \qquad \text{(ii)}: \frac{ S_{M,\delta}}{\E S_{M,\delta}}  
	\longrightarrow 1 \quad  \text{in probability as} \quad   M\to \infty.  \label{031720}
\end{align}

In order to prove \eqref{031720} (i), we recall \eqref{eq:righttail}, and see that, for any $\epsilon >0$, 
\begin{align*}
	\E S_{M,\delta}=\sum_{k=M-M^{1-\delta}}^M \Prob[\mathfrak{F}_k] &\gtrsim \sum_{k=M-M^{1-\delta}}^M N_k^{-\frac{2\beta}{3} (1+ \epsilon)c^{3/2}}  \\
	&\sim \sum_{k=M-M^{1-\delta}}^M k^{-\alpha \frac{2\beta}{3} (1 + \epsilon)c^{3/2}}\sim M^{1-\alpha\frac{2\beta}{3}c^{3/2}(1 + \epsilon)-\delta} \,. 
\end{align*} 
This diverges, $\E S_{M, \delta} \to \infty$ as $M \to \infty$, for any $c<c_*$, if we choose $\alpha=\alpha(c)$ to be sufficiently close to $3$ and $\delta=\delta(c)$ and $\epsilon = \epsilon(c)$ sufficiently close to $0$.  Hence, we have shown \eqref{031720}. 

We now prove \eqref{031720} (ii), for which it suffices to show 
\begin{align}
	\frac{ S_{M,\delta}-\E S_{M,\delta}}{\E  S_{M,\delta}}
		\longrightarrow 0, \quad \text{in probability as} \quad M\to \infty.  \label{031736}
\end{align}
To establish \eqref{031736}, we estimate the variance of $S_{M, \delta}$ from \eqref{eq:SMdelta} as
\begin{align}
	\mathrm{Var} \big( S_{M,\delta}\big)
		\leq \E S_{M,\delta}+2 \hspace{-2mm}\sum_{i=M-M^{1-\delta}}^M \sum_{ j=i+1}^{M} \left|\Prob \left[ \mathfrak{F}_i\cap\mathfrak{F}_j\right]-\Prob \left[ \mathfrak{F}_i\right] \Prob\left[\mathfrak{F}_j\right]\right|.  \label{031730}
\end{align}
Note that for the matrix sizes $N_i, N_j$ in the above double sum, we always have 
\begin{align*}
	N_j^{\frac{\alpha-1}{\alpha}}\lesssim N_j-N_i \lesssim N_j^{\frac{\alpha-\delta}{\alpha}} 
\end{align*}
and thus, since $\alpha>3$, 
\begin{align*}
	N_j^{\frac{2}{3}+\varepsilon_0}\lesssim N_j-N_i \lesssim N_j^{1-\varepsilon_0}.
\end{align*}
Hence, we can apply the decorrelation estimate \eqref{eq:rightdecor}, which, plugged into \eqref{031730} yields 
\begin{align}\label{varest}
	\text{Var}( S_{M,\delta}) \leq \E[ S_{M,\delta}]+CN_M^{-\varepsilon} \sum_{i=M-M^{1-\delta}}^M \sum_{ j>i}^{M} \Prob \big[ \mathfrak{F}_i\big] \Prob\big[\mathfrak{F}_j\big].
\end{align}
Both terms in the rhs.~are much smaller than $(\E[ S_{M,\delta}])^2$ (by an inverse power of $N_M$) and we thus conclude \eqref{031736} by  Chebyshev's inequality. This completes  the proof of \eqref{031737} by the arbitrariness of $c<c_*$.

\subsubsection{Upper bound for $\liminf$}  In this section, we prove that almost surely
\begin{equation} \label{eq:uppliminf}
	\liminf_{N\to \infty} \frac{\lambda_1^{(N)}}{(\log N)^{1/3}}\leq -c_{**}, \quad \text{with} \quad c_{**}:=\left(\frac{8}{\beta}\right)^{1/3}.
\end{equation}
The argument is similar to \eqref{031737} and we will hence be quite brief. To establish \eqref{eq:uppliminf}, we shall show that for any $c<c_{**}$, almost surely, ${\lambda_1^{(N)}}/{(\log N)^{1/3}}\leq -c$ occurs infinitely many times. Again, it suffices to show that this is true for the subsequence $N_k$ defined in \eqref{031740} with some $\alpha=\alpha(c)$ to be determined below. 

Similarly to Section \ref{subsubsec:lowlimsup}, for a fixed $0 < c<c_*$, we define (analogously to \eqref{eq:tailevents}) the event 
\begin{align*}
\mathfrak{E}_k = \mathfrak{E}_k(c):=\Big\{ \lambda_1^{(N_k)}\leq -c(\log N_k)^{1/3} \Big\}
\end{align*} 
and the partial sums 
\begin{align*}
	S_M:=\sum_{k=1}^M\mathbf{1}({\mathfrak{E}}_k), \qquad  S_{M,\delta}:=\sum_{k=M-M^{1-\delta}}^M\mathbf{1}({\mathfrak{E}}_k).
\end{align*}
Again, the task boils down to show that $S_{M,\delta}\to \infty$ in probability. Analogously to \eqref{031720}, this can be reduced to (i) a lower bound on the expectation, and (ii) an upper bound on the variance. Now, for the upper bound on the $\liminf$, we need the left tail estimate \eqref{eq:lefttail} and the left tail decorrelation estimate \eqref{eq:leftdecor}. The rest of the proof is completely analogous to Section \ref{subsubsec:lowlimsup} and thus omitted.

\subsection{Martingale inequalities: Upper bound for $\limsup$ and lower bound for $\liminf$} \label{subsec:maxmart}
The principal strategy in this part of the proof is as follows: We first establish the desired estimates (upper bound for $\limsup$ and lower bound for $\liminf$) for a suitable subsequence of $\lambda_1^{(N_k)}$'s with the aid of the tail probability estimates from Proposition \ref{prop:tails}. Afterwards, we extend these bounds to the whole sequence of $\lambda_1^{(N)}$'s by expressing the difference of $\lambda_1$'s as martingales and using suitable martingale inequalities. 
\subsubsection{Upper bound for $\limsup$} \label{subsubsec:upplimsup}
In this section, we show that, almost surely, 
\begin{align}
	\limsup_{N\to\infty} \frac{\lambda_1^{(N)}}{(\log N)^{2/3}}\leq c_*, \quad \text{with} \quad c_*:=\left(\frac{1}{2\beta}\right)^{2/3}.   \label{031102}
\end{align}
That means, we shall show that for any fixed $c>c^*$, almost surely, $\lambda_1^{(N)}\geq c(\log N)^{2/3}$ occurs only finitely many times.

\smallskip

\emph{Step 1: Upper bound for a subsequence.} Let $\alpha<3$ to be chosen and we define\footnote{Owing to the different exponent, this sequence $N_k$ is slightly different from the one defined in~\eqref{031740}, but
it plays the same role.}
\begin{align}
	N_k=\lceil k^\alpha\rceil. \label{subseq}
\end{align} 
By the one point tail probability estimate \eqref{eq:righttail}, we have 
\begin{align*}
	\Prob\big[\lambda_1^{(N_k)}\geq c(\log N_k)^{2/3}\big]\lesssim N_k^{-\frac{2\beta}{3} (1-\epsilon)c^{3/2}}\sim k^{-\alpha \frac{2\beta}{3}(1-\epsilon)c^{3/2}}
\end{align*}
for any $\epsilon > 0$. 
As $c>c_*$, we can choose $\alpha=\alpha(c)$ to be sufficiently close to $3$, so that 
\begin{align*}
	\sum_k\Prob\big[\lambda_1^{(N_k)}\geq c(\log N_k)^{2/3}\big]\lesssim \sum_k  k^{-\alpha \frac{2\beta}{3}(1-\epsilon)c^{3/2}}<\infty
\end{align*}
when $\epsilon$ is taken sufficiently small. 
By the Borel-Cantelli lemma, we infer that, almost surely,
\begin{align}
	\limsup_{k \to\infty} \frac{\lambda_1^{(N_k)}}{(\log N_k)^{2/3}}\leq c.  \label{031101}
\end{align}
\newline
\emph{Step 2: Extension to the whole sequence.} Next, we extend the conclusion from the subsequence $\{N_k\}$ to the whole $\N$. To this end, for $k \in \N$ and $\delta > 0$, we define
\begin{align}
	\mathcal{E}_k(\delta):=\Big\{ \exists n \in [ N_{k-1}+1, N_k ]: \lambda_1^{(n)}\geq (c+\delta)(\log n)^{2/3}\Big\}\bigcap \Big\{\lambda_1^{(N_k)}\leq c(\log N_k)^{2/3}\Big\}. \label{031113}
\end{align}
Our claim will follow from the following lemma, whose proof is presented below. 
\begin{lemma}[Extension lemma] \label{lem:ext}
For any $\delta>0$ it holds that
\begin{align}
	\sum_{k=1}^{\infty}\Prob\big[ \mathcal{E}_k(\delta)\big]<\infty.  \label{031103}
\end{align}
\end{lemma}
Indeed, Lemma \ref{lem:ext} together with the Borel-Cantelli lemma shows that, almost surely, only finitely many of the events $\mathcal{E}_k(\delta)$ can occur. Together with \eqref{031101}, we will have almost surely
\begin{align*}
	\limsup_{N \to\infty} \frac{\lambda_1^{(N)}}{(\log N)^{2/3}}\leq c+\delta.
\end{align*}
Due to the arbitrariness of $\delta>0$ and $c>c_*$, we conclude \eqref{031102} and are hence left with proving Lemma~\ref{lem:ext}. 
\begin{proof}[Proof of Lemma \ref{lem:ext}] 
Note that, since $\alpha<3$, there exists an $\varepsilon=\varepsilon(\alpha)$ such that 
\begin{align}
	0<N_{k}-N_{k-1}\leq N_k^{\frac23-\varepsilon}.  \label{031110}
\end{align}
\emph{Part 1: Lower bound on eigenvalue differences.} In the sequel, let $n\in [N_{k-1}+1, N_k]$. By eigenvalue rigidity \cite[Theorem 7.6]{loc_sc_gen}, we have (recall that $\widetilde{\lambda}_1^{(n)}$ is the largest eigenvalue of $H^{(n)}$ from \eqref{eq:Hdef})
\begin{align}
	\widetilde{\lambda}_1^{(n)}-\widetilde{\lambda}_1^{(n-1)}= \widetilde{\lambda}_1^{(n)}-\sqrt{\frac{n-1}{n}} \widetilde{\lambda}_1^{(n-1)}-\frac{1}{n}(1+\mathcal{O}_\prec(n^{-2/3})) \label{030801}
\end{align}
We wrote the difference in this form since the spectra of $H^{(n)}$ and $\sqrt{(n-1)/n}H^{(n-1)}$ are interlacing
 since the latter is a minor of the former, in fact \nc
for any $n\geq 2$, we can write 
\begin{align*}
	H^{(n)}=\left( \begin{array}{cc}
		\sqrt{\frac{n-1}{n}}H^{(n-1)}  & \bm{a}^{(n)}\\ \\
		(\bm{a}^{(n)})^*   & h_{nn}^{(n)}
	\end{array}\right).
\end{align*}
Note that here we rewrite the $n$-th diagonal entry of $H$ as $h_{nn}^{(n)}$.  Further, we denote the 
normalized  eigenvector associated with $\widetilde{\lambda}_1^{(n)}$ by $\bm{w}^{(n)}_1$ and we write
\begin{align}
	\bm{w}^{(n)}_1=\left(
	\begin{array}{c}
			\bm{v}^{(n)}_1 \\   \\
	u^{(n)}_1
	\end{array}
	\right), \label{19071010}
\end{align}
where $u^{(n)}_1$ is the $n$-th component of $\bm{w}^{(n)}_1$.  From the eigenvalue equation $H^{(n)} \bm{w}^{(n)}_1= \widetilde{\lambda}^{(n)}_1 \bm{w}^{(n)}_1$, 
we get 
\begin{equation} \label{19070901}
	\begin{split}
		&h^{(n)}_{nn} u^{(n)}_1+ (\bm{a}^{(n)})^* \bm{v}^{(n)}_1= \widetilde{\lambda}^{(n)}_1 u^{(n)}_1,\\
		& u^{(n)}_1 \bm{a}^{(n)}+ \sqrt{\frac{n-1}{n}}H^{(n-1)} \bm{v}^{(n)}_1= \widetilde{\lambda}^{(n)}_1 \bm{v}^{(n)}_1.  
	\end{split}
\end{equation}
By simple algebra, similarly, for example, to \cite[Eqs.~(7.3)-(7.5)]{minor}, we find
\begin{equation} \label{eq:basicrep}
	\widetilde{\lambda}_1^{(n)}=h_{nn}^{(n)}+\frac{1}{n}\sum_{\alpha=1}^{n-1}\frac{|\xi_\alpha^{(n)}|^2}{\widetilde{\lambda}_1^{(n)}-\sqrt{\frac{n-1}{n}}\widetilde{\lambda}_\alpha^{(n-1)}},
\end{equation}
where 
\begin{align}
	\xi_\alpha^{(n)}:= \sqrt{n}(\bm{w}^{(n-1)}_\alpha)^*\bm{a}^{(n)} \label{def of xi}
\end{align}
and $\bm{w}^{(n-1)}_\alpha$ is the $\ell^2$-normalized eigenvector of $H^{(n-1)}$ corresponding to the eigenvalue $\widetilde{\lambda}_\alpha^{(n-1)}$.  Note that $|\xi_{\alpha}^{(n)}|\prec 1$,
using the moment bounds \eqref{eq:Xdefmoments}  on the matrix elements,
and  $\mathbf{E} [ |\xi_{\alpha}^{(n)}|^2 \; | \mathcal{F}_{n-1}]=1$, where
 $ \mathcal{F}_m$ is the sigma field generated by the entries of $H^{(m)}$.
Actually, using the delocalisation of the eigenvectors $\| \bm{w}^{(n-1)}_\alpha\|_\infty\prec
N^{-1/2}$, 
one can easily see that $|\xi_{\alpha}^{(n)}|^2$ is asymptotically
$\chi^2$-distributed, more precisely it is
$\chi^2(2)/2$ (complex case) or $\chi^2(1)$ (real case), however we will not need this precise information.
\nc

Together with eigenvalue rigidity \cite[Theorem 7.6]{loc_sc_gen}, \eqref{eq:basicrep} implies 
\begin{equation}\label{mainid}
	\begin{split}
	2 &=\frac{1}{n}\frac{|\xi_{\alpha}^{(n)}|^2}{\widetilde{\lambda}_1^{(n)}-\sqrt{\frac{n-1}{n}} \widetilde{\lambda}_1^{(n-1)}}+\frac{1}{n}\sum_{\alpha= 2}^{n-1}\frac{|\xi_\alpha^{(n)}|^2}{\widetilde{\lambda}_1^{(n)}-\sqrt{\frac{n-1}{n}}\widetilde{\lambda}_\alpha^{(n-1)}}+\mathcal{O}_\prec(n^{-1/2})\\
&\geq \frac{1}{n}\frac{|\xi_{1}^{(n)}|^2}{\widetilde{\lambda}_1^{(n)}-\sqrt{\frac{n-1}{n}} \widetilde{\lambda}_1^{(n-1)}}+\frac{1}{n}\sum_{\alpha= \lfloor N^{\xi}\rfloor}^{n-1}\frac{|\xi_\alpha^{(n)}|^2}{\widetilde{\lambda}_1^{(n)}-\sqrt{\frac{n-1}{n}}\widetilde{\lambda}_\alpha^{(n-1)}}+\mathcal{O}_\prec(n^{-1/2})\\
&= \frac{1}{n}\frac{|\xi_{1}^{(n)}|^2}{\widetilde{\lambda}_1^{(n)}-\sqrt{\frac{n-1}{n}} \widetilde{\lambda}_1^{(n-1)}}+1+\mathcal{O}_\prec(n^{-\frac13+C\xi})
	\end{split}
\end{equation}
for some sufficiently small  $\xi>0$. Here in the last step we used the facts 
\begin{align} \label{eq:mainid2}
\frac{1}{n}\sum_{\alpha= \lfloor N^{\xi}\rfloor}^{n-1}\frac{|\xi_\alpha^{(n)}|^2}{\widetilde{\lambda}_1^{(n)}-\sqrt{\frac{n-1}{n}}\widetilde{\lambda}_\alpha^{(n-1)}}
+\frac{1}{n}\sum_{\alpha= 1}^{n-1}\frac{|\xi_\alpha^{(n)}|^2}{\sqrt{\frac{n-1}{n}}\widetilde{\lambda}_\alpha^{(n-1)}-2-\ii n^{-2/3+\xi}}=O_\prec(n^{-\frac13+C\xi})
\end{align} 
and 
\begin{align} \label{eq:mainid3}
&\frac{1}{n}\sum_{\alpha= 1}^{n-1}\frac{|\xi_\alpha^{(n)}|^2}{\sqrt{\frac{n-1}{n}}\widetilde{\lambda}_\alpha^{(n-1)}-2-\ii n^{-\frac23+\xi}} =-\Big[G_{nn}^{(n)}\big(2+\ii n^{-\frac23+\xi}\big)\Big]^{-1}+h_{nn}^{(n)}-2-\ii n^{-\frac23+\xi}\notag\\
&\qquad=-\big[m_{\rm sc}(2+\ii n^{-\frac23+\xi})\big]^{-1}-2+O_\prec(n^{-\frac13+C\xi})=-1+O_\prec(n^{-\frac13+C\xi}),
\end{align}
which follow from rigidity (\cite[Theorem 7.6]{loc_sc_gen}) and local law for the entries of Green function $G^{(n)}(z)=(H^{(n)}-z)^{-1}$ (\cite[Theorem 6.7]{erdHos2017dynamical}).
From this we conclude
\begin{align*}
	\widetilde{\lambda}_1^{(n)}-\sqrt{\frac{n-1}{n}} \widetilde{\lambda}_1^{n-1}\geq \frac{1}{n} |\xi_1^{(n)}|^2\Big(1+\mathcal{O}_\prec(n^{-1/3+C\xi})\Big).
\end{align*}
Plugging it into (\ref{030801}) we obtain 
\begin{align}
	\widetilde{\lambda}_1^{(n)}-\widetilde{\lambda}_1^{(n-1)}\geq \frac{1}{n} \Big(|\xi_1^{(n)}|^2-1\Big)+\mathcal{O}_\prec(n^{-4/3+C\xi}).  \label{031710}
\end{align}
Summing up this relation, together with (\ref{031110}),  implies that for any $n\in [N_{k-1}+1, N_k]$, 
\begin{align}\label{summ}
	\widetilde{\lambda}_1^{(N_k)}-\widetilde{\lambda}_1^{(n)}\geq \sum_{\ell=n+1}^{N_k} \frac{1}{\ell } \Big(|\xi_1^{(\ell )}|^2-1\Big)+\mathcal{O}_\prec(N_{k}^{-2/3-\varepsilon/2})
\end{align}
if we choose $\xi$ sufficiently small. 
Then, using the shifting and rescaling from \eqref{eq:TWconv}, and the above lower bound, we have 
\begin{align}\label{diff}
	\lambda_1^{(N_k)}-\lambda_1^{(n)} &= N_k^{2/3}(\widetilde{\lambda}_1^{(N_k)}-2)- n^{2/3}(\widetilde{\lambda}_1^{(n)}-2)=N_{k}^{2/3}(\widetilde{\lambda}_1^{(N_k)}-\widetilde{\lambda}_1^{(n)})+\mathcal{O}_\prec(N_k^{-1/3-\varepsilon})\notag\\
	&\geq N_{k}^{2/3}\sum_{\ell=n+1}^{N_k} \frac{1}{\ell } \Big(|\xi_1^{(\ell )}|^2-1\Big)+\mathcal{O}_\prec(N_{k}^{-\varepsilon/2}) .
\end{align}
 Defining
\begin{equation} \label{eq:martdef}
	X_{k, n}:=N_{k}^{2/3}\sum_{\ell=N_{k-1}+1}^{n} \frac{1}{\ell } \Big(|\xi_1^{(\ell )}|^2-1\Big),
\end{equation}
we notice that $\{X_{k,N_{k-1}+1},\ldots, X_{k,N_k}\}$ is a martingale adapted to the filtration $\{\mathcal{F}_{N_{k-1}}, \ldots \mathcal{F}_{N_k-1}\}$, where recall that   $\mathcal{F}_m$ is the $\sigma$-field generated by the entries of $H^{(m)}$. 
We can summarize the estimates above in the form that
\begin{equation}\label{mart}
\lambda_1^{(N_k)}-\lambda_1^{(n)} \ge  X_{k,N_k}-X_{k,n} +\mathcal{O}_\prec(N_{k}^{-\varepsilon/2}),
\end{equation}
i.e. we can estimate $\lambda_1^{(N_k)}-\lambda_1^{(n)}$ by a martingale difference up to negligible error. 
\newline
\emph{Part 2: Bound on $\Prob[\mathcal{E}_k(\delta)]$ by martingale inequalities.} By the definition of $\mathcal{E}_k(\delta)$ in (\ref{031113}) and (\ref{mart}), we have 
\begin{equation}\label{031120}
	\begin{split}
	\mathbf{P}\big[\mathcal{E}_k(\delta)\big] &\leq \mathbf{P}\Big[\min_{n\in [N_{k-1}+1, N_k]} (\lambda_1^{(N_k)}-\lambda_1^{(n)})\leq -\frac{\delta}{2}(\log N_k)^{\frac23}\Big]\\
&\leq \mathbf{P}\Big[ \max_{n\in [N_{k-1}+1, N_k]} \Big|X_{k,N_k}-X_{k,n}\Big|\geq \frac{\delta}{3}(\log N_k)^{\frac23} \Big]+\mathcal{O}(N_k^{-D})\\
&\leq  \mathbf{P}\Big[  \big|X_{k,N_k}\big|\geq \frac{\delta}{6}(\log N_k)^{\frac23} \Big]+\mathbf{P}\Big[ \max_{n\in [N_{k-1}+1, N_k]} \big|X_{k,n}\big|\geq \frac{\delta}{6}(\log N_k)^{\frac23} \Big]+\mathcal{O}(N_k^{-D})\\[1mm]
&\leq \left(1+\Big(\frac{p}{p-1}\Big)^p\right)\left(\frac{\delta}{6}(\log N_k)^{\frac23}\right)^{-p} \mathbf{E} |X_{k,N_k}|^p+\mathcal{O}(N_k^{-D})
	\end{split}
\end{equation}
for any $D> 0$ and any $p>1$, by applying Markov's inequality and the standard  $L^p$ maximal inequality for submartingales.
  Then, by the   moment inequality of martingales \cite{dharmadhikari1968bounds} 
   and using \eqref{030801}, for any $p\geq 2$ we have
\begin{equation}  \label{eq:martest}
	\begin{split}
	\mathbf{E} |X_{k,N_k}|^p &\leq C_p (N_k-N_{k-1})^{p/2} \max_{n\in [N_{k-1}, N_k]} \mathbf{E} \Big|\frac{N_{k}^{2/3}}{n }\big(|\xi_1^{(n )}|^2-1\big)\Big|^{p}\\
&\leq C_pN_k^{-\frac{p\varepsilon}{2}} \max_{n\in [N_{k-1}, N_k]} \mathbf{E} \big| |\xi_1^{(n )}|^2-1\big|^{p}=\mathcal{O}(N_k^{-\frac{p\varepsilon}{2}}). 
	\end{split}
\end{equation}
Since $p$ can be chosen  arbitrarily large, combining \eqref{031120} and \eqref{eq:martest}, we infer
\begin{align}
	\mathbf{P}\big[\mathcal{E}_k(\delta)\big]=\mathcal{O}(N_k^{-D}) \label{031124}
\end{align}
when $k$ is sufficiently large. This immediately implies \eqref{031103}. 
\end{proof}

\subsubsection{Lower bound for $\liminf$}
In this section, we show that, almost surely,
\begin{align}
	\liminf_{N\to \infty} \frac{\lambda_1^{(N)}}{(\log N)^{1/3}}\geq -c_{**}, \qquad c_{**}=\Big(\frac{8}{\beta}\Big)^{1/3}. \label{031127}
\end{align}
That means, we shall show that for any $c>c_{**}$, almost surely, $\lambda_1^{(N)}\leq -c(\log N)^{1/3}$ occurs only finitely many times. The argument is analogous to Section \ref{subsubsec:upplimsup} and hence we will be very brief. 

Again, we first prove that the conclusion is true for the subsequence defined in (\ref{subseq}) with $\alpha<3$ to be chosen. By the one point tail probability estimate \eqref{eq:lefttail}, for any $\epsilon > 0$, we have 
\begin{align*}
	\mathbf{P}\Big[\lambda_1^{(N_k)}\leq -c (\log N_k)^{1/3}\Big]\lesssim N_k^{-\frac{\beta c^3}{24}(1 - \epsilon)}\sim k^{-\alpha\frac{\beta c^3}{24}(1 - \epsilon)}.
\end{align*}
As $c>c_{**}$, we can choose $\alpha$ to be sufficiently close to $3$ so that 
\begin{align*}
	\sum_{k}\mathbf{P}\Big[\lambda_1^{(N_k)}\leq -c(\log N_k)^{1/3}\Big]\lesssim \sum_{k} k^{-\alpha \frac{\beta c^3}{24}(1-\epsilon)}<\infty
\end{align*}
if $\epsilon $ is taken sufficiently small. 
By the Borel-Cantelli lemma, we hence know that almost surely 
\begin{align*}
	\liminf_{k\to \infty} \frac{\lambda_1^{(N_k)}}{(\log N_k)^{1/3}}\geq -c. 
\end{align*}

Next, we extend the conclusion from the subsequence $\{N_k\}$ to the whole $\mathbf{N}$. To this end, for $\delta > 0$ and $k \in \N$, we define
\begin{align*}
	\widehat{\mathcal{E}}_k(\delta):=\Big\{ \exists n\in [N_{k-1}+1, N_k]: \lambda_1^{(n)}\leq -(c+\delta)(\log n)^{1/3}\Big\}\bigcap \Big\{ \lambda_1^{(N_{k-1})}\geq -c(\log N_{k-1})^{1/3}\Big\}
\end{align*}
Notice that, unlike \eqref{031113}, here in the second event on the rhs., we used $\lambda_1^{(N_{k-1})}$ instead of $\lambda_1^{(N_k)}$ to adapt to the direction of an inequality  analogous to (\ref{diff}).  Completely analogously to \eqref{031103}, it suffices to show 
\begin{align}
	\sum_k \mathbf{P}\big[\widehat{\mathcal{E}}_k(\delta)\big]<\infty. \label{031125}
\end{align}
In order to do so, we note that, similarly to \eqref{031120} and using the notation \eqref{eq:martdef}, we now have 
\begin{align*}
	\mathbf{P}\big[\widehat{\mathcal{E}}_k\big] &\leq \mathbf{P}\Big[\min_{n\in [N_{k-1}+1, N_k]} (\lambda_1^{(n)}-\lambda_1^{({N_{k-1}})})\leq -\frac{\delta}{2}(\log N_k)^{\frac13}\Big] \notag\\
	&\leq \mathbf{P}\Big[ \max_{n\in [N_{k-1}+1, N_k]} \Big|N_{k}^{2/3}\sum_{\ell=N_{k-1}+1}^{n} \frac{1}{\ell } \Big(|\xi_1^{(\ell )}|^2-1\Big)\Big|\geq \frac{\delta}{3}(\log N_k)^{\frac13} \Big]+\mathcal{O}(N_k^{-D})\\
	&=\mathbf{P}\Big( \max_{n\in [N_{k-1}+1, N_k]} \big|X_{k,n}\big|\geq \frac{\delta}{3}(\log N_k)^{\frac13} \Big)+\mathcal{O}(N_k^{-D})
\end{align*}
for any $D > 0$. 
Next, analogously to \eqref{031124}, using the maximum inequality and moment inequality for martingales as in \eqref{031120}--\eqref{eq:martest}, we obtain
\begin{align*} 
	\mathbf{P}\big[\widehat{\mathcal{E}}_k(\delta)\big]=\mathcal{O}(N_k^{-D})
\end{align*}
when $k$ is sufficiently large. We thus conclude \eqref{031125} and hence the proof of \eqref{031127}. 

\subsection{Concluding the proof} Given \eqref{031737} and \eqref{031102} for the $\limsup$, and \eqref{eq:uppliminf} and \eqref{031127} for the $\liminf$, the proof of Theorem \ref{thm:main} is immediate. \qed

\section{Decorrelation estimate: Proof of Proposition \ref{prop:decor}} \label{app:decor}

The proof of Proposition \ref{prop:decor} is based on a dynamical argument, similar to our approach in \cite{minor}, but now we need to cover the moderate deviation regime. We compare the array $X$ with a Gaussian divisible array. To do so, we embed $X$ into the Ornstein-Uhlenbeck (OU) flow
\begin{equation}
	\dif X_t = -\frac{1}{2}X_t\dif t + \dif B_t,\quad X_0=X,
	\label{eq:GFT_flow}
\end{equation}
where $B_t$ is a doubly infinite matrix-valued Brownian motion with the same second order correlation structure as $X$, such that the first and second moments of $X_t$ are preserved along the flow \eqref{eq:GFT_flow}. The solution of the OU flow \eqref{eq:GFT_flow} is naturally understood by considering the analogous flow on $X_t^{(N)}$ for every finite $N \in \N$. The OU flow inserts an independent Gaussian component of variance $t$ in each matrix element.

For $t \ge 0$ and $N \in \N$, we naturally associate the time dependent Wigner matrix $H_t^{(N)} := N^{-1/2} X_t^{(N)}$, as in \eqref{eq:Hdef}, and denote the eigenvalues of $H_t^{(N)}$ by $\{\widetilde{\lambda}_i^{(N)}(t)\}_{i=1}^{N}$ in decreasing order. It is well-known that the eigenvalues do not cross along their evolution therefore $ \widetilde{\lambda}_1^{(N)}(t)$ remains
the top eigenvalue for all $t$.
The shifted and rescaled time-dependent top eigenvalue is denoted by 
\begin{equation*}
\lambda_1^{(N)}(t) := N^{2/3} (\widetilde{\lambda}_1^{(N)}(t) - 2),
\end{equation*}
just as in \eqref{eq:TWconv}. Moreover, for $N_1, N_2 \in \N$ and $x_1, x_2 \ge 0$, we introduce the time-dependent analogs of the tail events \eqref{eq:tailevents}:
\begin{equation} \label{eq:taileventstime}
	\mathfrak{F}_l(t, x_l) := \left\{ \lambda_{1}^{(N_l)}(t) \ge x_l (\log N_l)^{2/3}\right\}\,, \quad  \mathfrak{E}_l(t, x_l) := \left\{ \lambda_{1}^{(N_l)}(t) \le  - x_l (\log N_l)^{1/3}\right\}\,, \quad l = 1,2\,. 
\end{equation}

As the first step, we now show that, for sufficiently long times, the tail events in \eqref{eq:taileventstime} 
are decorrelated. 

\begin{proposition}[Step 1: Decorrelation for Gaussian divisible ensemble] \label{prop:DBM}
Fix $\epsilon_0 >0$ and take two large positive integers $N_1,  N_2 \in \N$ such that $N_2^{2/3 + \epsilon_0} \le N_2 -N_1 \le N_2^{1 - \epsilon_0}$. Fix a large constant $K > 0$. Then there exist constants $\delta, \omega >0$ such that 
\begin{subequations} \label{eq:decorDBM}
	\begin{equation} \label{eq:rightdecorDBM}
		\Prob \left[\mathfrak{F}_1(t, x_1) \cap \mathfrak{F}_2(t, x_2)\right] = \Prob\left[\mathfrak{F}_1(t, x_1)\right] \, \Prob\left[\mathfrak{F}_2(t, x_2)\right] \big(1 + \mathcal{O}(N_2^{-\delta})\big)
	\end{equation}
	and 
	\begin{equation} \label{eq:leftdecorDBM}
		\Prob \left[\mathfrak{E}_1(t, x_1) \cap \mathfrak{E}_2(t, x_2)\right] = \Prob\left[\mathfrak{E}_1(t, x_1)\right] \, \Prob\left[\mathfrak{E}_2(t, x_2)\right] \big(1 + \mathcal{O}(N_2^{-\delta})\big) 
	\end{equation}
	uniformly in $t \ge N_2^{-1/3+\omega}$ and $K^{-1} \le x_l \le K $, $l=1,2$. 
\end{subequations}
\end{proposition}

The proof of Proposition \ref{prop:DBM} is given in Section \ref{subsec:DBM} below, and it yields the decorrelation of the tail events for matrices with a Gaussian component of size $N^{-1/3+\epsilon}$. As the second step, we now remove this Gaussian component introduced in Proposition \ref{prop:GFT} via a \emph{Green function comparison} (GFT) argument. The proof of the following Proposition \ref{prop:GFT} is given in Section \ref{subsec:GFT}. 

\begin{proposition}[Step 2: Removing the Gaussian component] \label{prop:GFT}
Fix $\epsilon_0 >0$ and take two large positive integers $N_1,  N_2 \in \N$ such that $N_2^{2/3 + \epsilon_0} \le N_2 -N_1 \le N_2^{1 - \epsilon_0}$. Fix a large constant $K > 0$. Assume that there exist constants $\delta, \omega >0$ such that \eqref{eq:decorDBM} holds uniformly in $t \ge N_2^{-1/3 + \omega}$ and $K^{-1} \le x_l \le K$, for $l=1,2$. Then \eqref{eq:decorDBM} holds at $t=0$ as well, uniformly in $K^{-1} \le x_l \le K$, for $l=1,2$. 
\end{proposition}

Combining Propositions \ref{prop:DBM}--\ref{prop:GFT}, we immediately conclude the proof of Proposition~\ref{prop:decor}. 
\qed

\subsection{Dyson Brownian motion: Proof of Proposition \ref{prop:DBM}} \label{subsec:DBM}
The proof of Proposition \ref{prop:DBM} is based on the following result, which we prove below. 
\begin{proposition} \label{prop:DBMunderlying}
Fix $\epsilon_0 > 0	$ and take to large positive integers $N_1, N_2 \in \N$ such that $0 < N_2 - N_1 \le N_2^{1 - \epsilon_0}$ and a large constant $K > 0$. For $l =1, 2$, consider the time-dependent Wigner matrices $H_t^{(N_l)}$ and let $\{{\bm w}_i^{(N_l)}(t)\}_{i=1}^{N_l}$ be the $\ell^2$-normalized eigenvectors corresponding to the eigenvalues $\{\widetilde{\lambda}_i^{(N_l)}(t)\}_{i=1}^{N_l}$ in decreasing order. Assume that there exist small constants $\omega_*, \omega_E>0$ such that\footnote{Here, the "shorter" of the two vectors, $\bm w_i^{(N_1)} \in \C^{N_1}$ is understood to be extended to a vector in $\C^{N_2}$ by augmenting it with $(N_2 - N_1)$ many zeros.}
	\begin{equation} \label{eq:overlapbound}
		\big|\langle {\bm w}_i^{(N_1)}(t), {\bm w}_j^{(N_2)}(t) \rangle\big|\le N^{-\omega_E},
	\end{equation}
	simultaneously for all $t \ge 0$ and $1\le i,j\le N^{\omega_*}$. Then, there exist constants $\omega_1, \delta >0$ such that for any $D>0$ we have
	\begin{subequations} 		\label{eq:asympindep}
\begin{equation} 	\label{eq:asympindepa}
			\mathbf{P}\big[\mathfrak{F}_1(t,x_1)\cap \mathfrak{F}_2(t,x_2)\big]=\Prob\big[\mathfrak{F}_1(t,x_1)\big] \, \Prob\big[\mathfrak{F}_2(t,x_2)\big] \, \big(1+\mathcal{O}(N_2^{-\delta})\big),
\end{equation}
and
	\begin{equation} 	\label{eq:asympindepb}
		\mathbf{P}\big[\mathfrak{E}_1(t,y_1)\cap \mathfrak{E}_2(t,y_2)\big]=\Prob\big[ \mathfrak{E}_1(t,y_1)\big] \, \Prob\big[\mathfrak{E}_2(t,y_2)\big] \,  \big(1+\mathcal{O}(N_2^{-\delta})\big)
\end{equation}
	\end{subequations}
	uniformly in $t \ge N_2^{-1/3+\omega_1}$, and $K^{-1} \le x_l \le K$, $l =1,2$.
\end{proposition}

\begin{proof}[Proof of Proposition \ref{prop:DBM}]
If $N_2 - N_1 \ge N_2^{2/3 + \epsilon_0}$, we have from \cite[Proposition~3.1]{minor}, that \eqref{eq:overlapbound} is satisfied with $\omega_E = \epsilon_0/2$. Hence, by Proposition \ref{prop:DBMunderlying}, the decorrelation statements in \eqref{eq:asympindep} hold for $t \ge N_2^{-1/3 + \omega_1}$, which concludes the proof of Proposition \ref{prop:DBM}.
\end{proof} 

We are hence left with proving Proposition \ref{prop:DBMunderlying}.

\begin{proof}[Proof of Proposition \ref{prop:DBMunderlying}]
	We present only the proof of \eqref{eq:asympindepa}, since the proof of \eqref{eq:asympindepb} is completely analogous. It is well known (see, e.g., \cite{erdHos2017dynamical, anderson2010introduction}) that for any $N\in \N$ the eigenvalues of $H_t^{(N)}$ are the unique strong solution of the Dyson Brownian Motion (DBM):
		\begin{equation}
	\label{eq:DBM}
	\dif \widetilde{\lambda}_i^{(N)}(t)=\sqrt{\frac{\beta}{2N}}\dif b_i^{(N)}(t)+\frac{1}{N}\sum_{j\ne i}\frac{1}{\widetilde{\lambda}_i^{(N)}(t)-\widetilde{\lambda}_j^{(N)}(t)}\dif t -\frac{1}{2} \widetilde{\lambda}_i^{(N)}(t)\dif t
	\end{equation}
	with initial data given by the eigenvalues of $H_0^{(N)}=H^{(N)}$. Here, for fixed $N\in \N$ the $b_i^{(N)}(t)$ are standard real i.i.d. Brownian motions. However, for different $N_1,N_2$ they are correlated, and their covariation process is given by $|\langle {\bm w}_i^{(N_1)}(t), {\bm w}_j^{(N_2)}(t) \rangle|^2 \dif t$. Building on the analysis of weakly correlated DBMs from \cite[Section 7]{macroCLT_complex} (see also \cite[Section 4.5]{bourgade2024fluctuations}), in \cite[Proposition~5.1]{minor} it has been shown\footnote{We point out that in \cite{minor} we used a slightly different version of the DBM compared to \eqref{eq:DBM}. In fact, in \cite[Eq.~(3.4)]{minor}, we consider a DBM without the term $\widetilde{\lambda}_i^{(N)}(t)/2$. However, it can easily be seen that this term does not influence the analysis in \cite{macroCLT_complex}.} that under the assumption \eqref{eq:overlapbound} there exist two fully independent processes\footnote{It is easy to see that those two processes correspond to the top eigenvalues of two independent GUE/GOE matrices.} $\mu_i^{(1)}(t), \mu_i^{(2)}(t)$ satisfying \eqref{eq:DBM}, this time with independent (also among each other) driving Brownian motions $\{\beta_i^{(1)}(t)\}_{i\in [N_1]}$, $\{\beta_i^{(2)}(t)\}_{i\in [N_2]}$, \nc and a small $\omega>0$, such that
	\begin{equation}
		\label{eq:DBMminor}
		\big|\widetilde{\lambda}_1^{(N_l)}(t)-\mu_1^{(l)}(t)\big|\le N_2^{-2/3-\omega},
	\end{equation}
	with very high probability for $l=1,2$ and $t\ge N_2^{-1/3+\omega_1}$.
	
	To keep the notation simple, in the following we always work on the very high probability event on which \eqref{eq:DBMminor} holds, even if not explicitly written. This will lead to the additional $\mathcal{O}(N_2^{-D})$ errors in the formulas below. To further simplify notation, we now set 
	\begin{equation} \label{eq:tildexdef}
		\widetilde{x}_l := 2 + x_l N_l^{-2/3} (\log N_l)^{2/3} \quad \text{for} \quad l = 1,2\,. 
	\end{equation}
	Then, for $t\ge N_2^{-1/3+\omega_1}$ and for any $D>0$, we have
	\begin{equation}
	\label{eq:newh1}
		\begin{split}
			\mathbf{P}\big[ \mathfrak{F}_1(t,x_1)\cap \mathfrak{F}_2(t,x_2)\big] &=\mathbf{P}\left[\widetilde{\lambda}_1^{(N_1)}(t)\ge \widetilde{x}_1, \widetilde{\lambda}_1^{(N_2)}(t)\ge \widetilde{x}_2\right] \\
			&\le \, \mathbf{P}\left[\mu_1^{(1)}(t)\ge \widetilde{x}_1-\big|\widetilde{\lambda}_1^{(N_1)}(t)-\mu_1^{(1)}(t)\big|, \mu_1^{(2)}(t) \ge \widetilde{x}_2-\big|\widetilde{\lambda}_1^{(N_2)}(t)-\mu_1^{(2)}(t)\big|\right] \\
			&\le \, \mathbf{P}\left[\mu_1^{(1)}(t)\ge \widetilde{x}_1-N_2^{-2/3-\omega}, \mu_1^{(2)}(t)\ge \widetilde{x}_2-N_2^{-2/3-\omega}\right]+\mathcal{O}(N_2^{-D}) \\
			&= \, \mathbf{P}\left[\mu_1^{(1)}(t)\ge \widetilde{x}_1-N_2^{-2/3-\omega}\right] \, \mathbf{P}\left[\mu_1^{(2)}(t)\ge \widetilde{x}_2-N_2^{-2/3-\omega}\right]+\mathcal{O}(N_2^{-D}),
		\end{split}
	\end{equation}
	where in the last equality we used that $\mu_1^{(1)}(t)$ and $\mu_1^{(2)}(t)$ are independent. 
	Proceeding analogously, we also obtain a bound in the opposite direction
	\begin{equation}\label{eq:newh15}
	\mathbf{P}\big[ \mathfrak{F}_1(t,x_1)\cap \mathfrak{F}_2(t,x_2)\big]
	\ge \, \mathbf{P}\left[\mu_1^{(1)}(t)\ge \widetilde{x}_1+N_2^{-2/3-\omega}\right] \, \mathbf{P}\left[\mu_1^{(2)}(t)\ge \widetilde{x}_2+N_2^{-2/3-\omega}\right]+\mathcal{O}(N_2^{-D}).
	\end{equation}
	Similar bounds on the probability of a single event $\mathfrak{F}_l(t,x_l)$, $l=1,2$, 
	directly follow from \eqref{eq:DBMminor}:
		\begin{equation}
		\label{eq:newh2}
		\begin{split}
			\mathbf{P}\big[ \mathfrak{F}_l(t,x_l)\big] \le\mathbf{P}\left[\mu_1^{(l)}(t)\ge \widetilde{x}_l-N_2^{-2/3-\omega}\right]+\mathcal{O}(N_2^{-D}), \\
			\mathbf{P}\big[ \mathfrak{F}_l(t,x_l)\big] \ge\mathbf{P}\left[\mu_1^{(l)}(t)\ge \widetilde{x}_l+N_2^{-2/3-\omega}\right]+\mathcal{O}(N_2^{-D}).
			\end{split}
	\end{equation}
	Next, we show that the small shifts $\pm N_2^{-2/3-\omega}$ introduced by \eqref{eq:DBMminor} 
	passing from $\widetilde{\lambda}_1(t)$ to $\mu_1(t)$ can be removed at the cost of a multiplicative error $(1+o(1))$.
	This will follow from an analogous result for GOE/GUE matrices formulated in the next  lemma,
	whose proof is presented at the end of this section.
	\begin{lemma}
	\label{lem:logconcave}
	Fix any small $\omega>0$ and large $K>0$. Let $\mu_1$ be the largest eigenvalue of an $N\times N$ GOE/GUE matrix. Then, there exists a $\delta>0$ such that
	\begin{equation}
	\label{eq:smooth}
		\mathbf{P}\left[\mu_1\ge 2+\frac{x(\log N)^{2/3}}{N^{2/3}} + \mathcal{O}( N^{-2/3-\omega})\right]=\mathbf{P}\left[\mu_1\ge 2+\frac{x(\log N)^{2/3}}{N^{2/3}}\right]\big(1+\mathcal{O}(N^{-\delta})\big),
	\end{equation}
	uniformly in $x\in [K^{-1},K]$. A similar relation holds for the  left tail 
	event $\mathbf{P}\left[\mu_1\le 2- \ldots \right]$ as well.
	\end{lemma}
	Notice that by  \eqref{eq:newh2} it follows that \eqref{eq:smooth} holds for $\mu_1$ replaced by $\mu_1(t)$ or by 
	$\widetilde{\lambda}_1^{(N)}(t)$ as well, if $t\ge N^{-1/3+\omega}$ since $\mu_1=\mu_1(t=\infty)$. The additive $\mathcal{O}(N^{-D})$ errors can be
	incorporated into the leading term using that all probabilities are of polynomial order in $N$ thanks to 
	the bounds in Proposition~\ref{prop:tails}. 
	 Using \eqref{eq:smooth} for $\mu^{(l)}(t)$ in \eqref{eq:newh1} and \eqref{eq:newh2}, \nc we thus conclude 
	\begin{equation*}
		\mathbf{P}\big[\mathfrak{F}_1(t,x_1)\cap \mathfrak{F}_2(t,x_2)\big]\le\mathbf{P}\big[ \mathfrak{F}_1(t,x_1)\big] \, \mathbf{P}\big[ \mathfrak{F}_2(t,x_2)\big] \, \big(1+\mathcal{O}(N_2^{-\delta})\big).
	\end{equation*}
	Proceeding analogously, but using \eqref{eq:newh15} instead of \eqref{eq:newh1}, \nc we also obtain  a matching lower bound
	\begin{equation*}
	\mathbf{P}\big[\mathfrak{F}_1(t,x_1)\cap \mathfrak{F}_2(t,x_2)\big]\ge\mathbf{P}\big[ \mathfrak{F}_1(t,x_1)\big] \, \mathbf{P}\big[ \mathfrak{F}_2(t,x_2)\big] \, \big(1+\mathcal{O}(N_2^{-\delta})\big),
\end{equation*}
	which concludes the proof of Proposition \ref{prop:DBMunderlying}.
\end{proof}

\begin{proof}[Proof of Lemma~\ref{lem:logconcave}]

For some  $\tau_N>0$  of order $N^{-2/3-\omega}$, we define the  ($N$-dependent) intervals
\begin{align}
\mathcal{A}:=\Big[2+\frac{x(\log N)^{2/3}}{N^{2/3}}-\tau_N,\infty\Big), \qquad \mathcal{B}:=\Big[2+\frac{x(\log N)^{2/3}}{N^{2/3}}+N^{-2/3}, \infty\Big).\label{093002}
\end{align}
 Let $c_N$ be the solution of the equation
\begin{align}
-(1-c_N\tau_N)+N^{-2/3}c_N=0, \label{093001}
\end{align} 
then clearly  $c_N=N^{2/3}(1+o(1))$.
Let $s_N:=(1-c_N\tau_N)\in (0,1)$. By (\ref{093001}) and the definition in (\ref{093002}), we have 
\begin{align*}
s_N\mathcal{A}+(1-s_N) \mathcal{B}=\mathcal{C}:=\Big[2+\frac{x(\log N)^{2/3}}{N^{2/3}},\infty\Big).
\end{align*}
We then use the fact that for GOE/GUE, the distribution of $\mu_1$ is log-concave, which is a consequence of the Pr\'{e}kopa-Leindler inequality and the fact that the joint density of the eigenvalues of GOE/GUE is log-concave; see, e.g. Theorem 5 of \cite{log_concavity24}.
Particularly, we have 
\begin{align*}
\mathbf{P}(\mu_1\in \mathcal{C})\geq \big(\mathbf{P}(\mu_1\in \mathcal{A})\big)^{s_N}\big(\mathbf{P}(\mu_1\in \mathcal{B})\big)^{1-s_N}=\mathbf{P}(\mu_1\in \mathcal{A}) \big(\mathbf{P}(\mu_1\in \mathcal{A})\big)^{-c_N\tau_N} \big(\mathbf{P}(\mu_1\in \mathcal{B})\big)^{c_N\tau_N}.
\end{align*}
From the tail probability bounds  for GOE/GUE from
Proposition~\ref{prop:tails} we have that both $\mathbf{P}(\mu_1\in \mathcal{A})$ and $\mathbf{P}(\mu_1\in \mathcal{B})$
are of order $N^{-c(x)}$ with a constant $c(x)\sim 1$ depending on $x$. Since
$c_N\tau_N\sim N^{-\omega}$, the $c_N\tau_N$ powers of these probabilities are very close to 1, thus we have
\begin{align*}
\mathbf{P}(\mu_1\in \mathcal{C})\geq \mathbf{P}(\mu_1\in \mathcal{A})(1+\mathcal{O}(N^{-\delta}))
\end{align*}
for some $\delta>0$. By monotonicity, the opposite inequality $\mathbf{P}(\mu_1\in \mathcal{C})\leq \mathbf{P}(\mu_1\in \mathcal{A})$
is straightforward. Thus we conclude the proof of (\ref{eq:smooth}) in case the shift
$\mathcal{O}(N^{-2/3-\omega})$ is negative. In case the shift is positive, we only need to switch the role of the LHS and the RHS of (\ref{eq:smooth}). Hence, we completed the proof. 
\end{proof}

\nc

\subsection{Green function comparison: Proof of Proposition \ref{prop:GFT}} \label{subsec:GFT}
We focus on proving that \eqref{eq:rightdecorDBM} holds for $t = 0$; the arguments for \eqref{eq:leftdecorDBM} are analogous and thus omitted. Our approach follows the comparison strategy introduced in \cite{small_dev} for the moderate deviations of the largest eigenvalue of a Wigner matrix, later adapted to the case of multiple matrices in \cite{univ_extr}. As in  \cite[Eq. (2.11)]{small_dev}, we fix a small $\varepsilon>0$ and rewrite the lhs. of \eqref{eq:rightdecorDBM} as
\begin{equation} \label{eq:expresp}
	\mathbf{P}\left[ \mathfrak{F}_1(t, x_1)\cap\mathfrak{F}_2(t, x_2)\right] = \E \left[ F\big(\chi_{E_1}\big(H_t^{(N_1)}\big)\big) F\big(\chi_{E_2}\big(H_t^{(N_2)}\big)\big)\right] + \mathcal{O}(N_2^{-D}),
\end{equation}
where we used the notation (recall \eqref{eq:tildexdef}) 
\begin{equation*}
	\chi_{E_l}(\cdot) := \mathrm{Tr}\big( \bm{1}_{[E_l,E_L]} (\cdot)\big),\quad \text{with} \quad E_l:=\widetilde{x}_l,\quad E_L:=2+N^{-2/3+\varepsilon},\quad l=1,2,
\end{equation*}
where $F:\R_+\to \R_+$ is a smooth non-decreasing cut-off function such that $F(x)=0$ for $x\in [0, 1/9]$ and 
$F(x)=1$ for $x\ge 2/9$, as in \cite[Eq.~(2.12)]{small_dev}.

Next, we wish to approximate $\chi_{E_l}\big(H_t^{(N_l)}\big)$ by a function, that can be expressed in terms of the resolvent 
\begin{equation*}
	G_{l,t}(z):=\big(H^{(N_l)}_t-z\big)^{-1}\quad \text{for} \quad z\in\C\setminus\R \,, 
\end{equation*}
which will allow us to prove Proposition \ref{prop:GFT} via a \emph{Green function comparison} (GFT) argument. 
In fact, applying \cite[Eq.~(2.16)]{small_dev} and using that $\chi_{E_l}\big(H_t^{(N_l)}\big)$ is integer-valued, we obtain 
\begin{equation} \label{eq:exptoresolvent}
	\begin{split}
	\E \left[ F\big(\chi_{E_1+\ell_1}*\theta_{\eta_1}\big(H_t^{(N_1)}\big)\big) \right. &\left.  F\big(\chi_{E_2+\ell_2}*\theta_{\eta_2}\big(H_t^{(N_2)}\big)\big)\right] - \mathcal{O}(N_2^{-D})\\
& \hspace{-1cm}\le \E \left[ F\big(\chi_{E_1}\big(H_t^{(N_1)}\big)\big) \, F\big(\chi_{E_2}\big(H_t^{(N_2)}\big)\big) \right]\\
& \hspace{-1cm}\le \E \left[ F\big(\chi_{E_1-\ell_1}*\theta_{\eta_1}\big(H_t^{(N_1)}\big)\big) \,  F\big(\chi_{E_2-\ell_2}*\theta_{\eta_2}\big(H_t^{(N_2)}\big)\big)\right] + \mathcal{O}(N_2^{-D}),
	\end{split}
\end{equation}
where $\ell_l:=N_l^{-2/3-\varepsilon/9}$, $\eta_l:=N_l^{-2/3-\varepsilon}$
 and $\theta_\eta(x): =\eta/[\pi(x^2+\eta^2)]$ is a Cauchy mollifier. \nc   Following \cite{small_dev}, we can then write
    \begin{equation*}
	\mathcal{X}_{E,l,t}:= \chi_{E}*\theta_{\eta_l}\big(H^{(N_l)}_t\big)=\frac{N}{\pi}\int_{E}^{E_L}\Im \langle G_{l,t}(y+\ii\eta_l)\rangle \dif y,\quad l=1,2.
\end{equation*}
Now we are ready to state our main technical result,  to be proven at the end
 of this section,  which is \nc needed for the proof of Proposition \ref{prop:GFT}. 
\begin{proposition}\label{prop:R_dif} \label{prop:GFTunderlying} Using the above notations, we define
	\begin{equation}
		R_t=R_t(E_1,E_2):=F(\X_{ E_1, \nc 1,t})F(\X_{ E_2, \nc 2,t}).
		\label{eq:def_R}
	\end{equation}
	Then for any $t\ge 0$ it holds that
	\begin{equation}
		\vert\partial_t \E R_t(E_1,E_2)\vert\lesssim N_2^{-1/6+3\varepsilon} \E R_t(E_1-2\ell_1,E_2-2\ell_2)+ \mathcal{O}(N_2^{-D}).
		\label{eq:dif_R_aim}
	\end{equation}
\end{proposition}
We remark that the \nc estimate \eqref{eq:dif_R_aim} is the analogue of \cite[Eq. (2.57)]{small_dev}. Fix a time $t$ such that $N_2^{-1/3+\omega} \le t\lesssim 1$ and integrate \eqref{eq:dif_R_aim} from $s$ to $t$ for any $s\in [0,t]$:
\begin{equation}
\left\vert\E R_t(E_1,E_2) - \E R_s(E_1,E_2)\right\vert \lesssim N_2^{-1/6+3\varepsilon}\sup_{r\in [s,t]}\left\vert \E R_r(E_1-2\ell_1,E_2-2\ell_2)\right\vert + N_2^{-D}. 
\label{eq:for_R_iteration}
\end{equation} 
Using \eqref{eq:for_R_iteration} we prove that for any $k\in\N$, independent of $N_1,N_2$, it holds for any $r\in[0,t]$ that
\begin{equation}
\left\vert\E R_t(E_1,E_2) - \E R_r(E_1,E_2)\right\vert \lesssim \sum_{m=1}^{k-1} N_2^{(-1/6+3\varepsilon)m}\left\vert \E R_t(E_1-m\cdot 2\ell_1, E_2-m\cdot 2\ell_2)\right\vert + N_2^{(-1/6+3\varepsilon)k}.
\label{eq:R_iteration}
\end{equation} 
To obtain \eqref{eq:R_iteration} for $k=1$ we combine \eqref{eq:for_R_iteration} with the trivial bound $|R_r(E_1-2\ell_1,E_2-2\ell_2)|\le 1$ coming from $|F|\le 1$. We now use induction on $k$: from \eqref{eq:R_iteration} applied to $k=n$ and $E_1-2\ell_1, E_2-2\ell_2$ instead of $E_1,E_2$, we get
$$
  |\E R_r(E_1-2\ell_1,E_2-2\ell_2)| \le  |\E R_t(E_1-2\ell_1,E_2-2\ell_2)| + \sum_{m=1}^{n-1} (\ldots)
  +N_2^{(-1/6+3\varepsilon)n},
 $$
 which is then used to upper bound the rhs. of \eqref{eq:for_R_iteration} to arrive at  \eqref{eq:R_iteration} for $k=n+1$.
 
Next we show that for any fixed $m\in\N$ we have
\begin{equation}
\E R_t(E_1-m\cdot 2\ell_1,E_2-m\cdot 2\ell_2) =\E R_t(E_1,E_2)(1+\mathcal{O}(N^{-\delta_1}_2))
\label{eq:R_shift}
\end{equation}
for some $\delta_1>0$ which is independent of $N_1, N_2$. 
 In order to prove~\eqref{eq:R_shift} we go back to $\mathfrak{F}_1\cap \mathfrak{F}_2$ since
the small shift in the argument can be controlled on the $\mathfrak{F}$ level thanks to Lemma~\ref{lem:logconcave}. Denote for short $\Delta x_l:=  m\cdot 2 \ell_l N_l^{2/3}(\log N_l)^{-2/3}$, $l=1,2$. We have from \eqref{eq:exptoresolvent} that 
\begin{equation}
\begin{split}
&\mathbf{P}\left[\mathfrak{F}_1(t, x_1+\Delta x_1)\cap \mathfrak{F}_{2}(t, x_2+\Delta x_2)\right]+\mathcal{O}(N_2^{-D})\\
&\quad\le \E R_t(E_1,E_2) \le \mathbf{P}\left[\mathfrak{F}_1(t, x_1-\Delta x_1)\cap \mathfrak{F}_{ 2}(t, x_2-\Delta x_2)\right]+\mathcal{O}(N_2^{-D}).
\end{split}
\label{eq:FF_to_PP} 
\end{equation}
Combining \eqref{eq:FF_to_PP} with the assumption that \eqref{eq:rightdecorDBM} holds at time $t$,  and replacing $x_l$
with $x_l\pm \Delta x_l$, we obtain
\begin{equation}
\begin{split}
&\mathbf{P}\left[\mathfrak{F}_1(t, x_1+\Delta x_1)\right]\mathbf{P}\left[\mathfrak{F}_{2}(t, x_2+\Delta x_2)\right]\left(1+\mathcal{O}(N^{-\delta}_2)\right)\\
&\quad\le \E R_t(E_1,E_2)\le \mathbf{P}\left[\mathfrak{F}_1(t, x_1-\Delta x_1)\right]\mathbf{P}\left[\mathfrak{F}_{ 2}(t, x_2-\Delta x_2)\right]\left(1+\mathcal{O}(N^{-\delta}_2)\right).
\end{split}
\label{eq:FF_mult_err}
\end{equation}
Here we absorbed the additive error terms in \eqref{eq:FF_to_PP} into the multiplicative error terms in the same way as it was discussed below the statement of Lemma \ref{lem:logconcave}. Recall from the same discussion that  \eqref{eq:smooth} is also valid for $\widetilde{\lambda}_1^{(N_l)}(t)$, $l=1,2$,  i.e.   $\mathbf{P}\left[\mathfrak{F}_l(t, x_l \pm \Delta x_l)\right]$
can be compared with $\mathbf{P}\left[\mathfrak{F}_l(t, x_l)\right]$ since  $\Delta x_l = \mathcal{O}(N^{-\omega})$. 
 Thus, applying Lemma \ref{lem:logconcave}  for $\widetilde{\lambda}_l^{(N_l)}(t)$ to the lower and upper bounds on $\E R_t(E_1,E_2)$ from \eqref{eq:FF_mult_err}, we get
\begin{equation}
\E R_t(E_1,E_2) = \mathbf{P}\left[\mathfrak{F}_1(t, x_1)\right]\mathbf{P}\left[\mathfrak{F}_{2}(t, x_2)\right]\left(1+\mathcal{O}(N^{-\delta_1}_2)\right)
\label{eq:R_factor}
\end{equation}
for some $\delta_1>0$. Using~\eqref{eq:R_factor} once for the arguments $(E_1, E_2)$ and once for $(E_1-m\cdot 2\ell_1, 
E_2-m\cdot 2\ell_2)$,
and controlling the shifts in the rhs. of~\eqref{eq:R_factor} via Lemma~\ref{lem:logconcave}, we immediately obtain
\eqref{eq:R_shift}.

Choosing sufficiently large $k$ in \eqref{eq:R_iteration} with $r=0$, and applying \eqref{eq:R_shift} to each of the terms in the rhs. of \eqref{eq:R_iteration}, we conclude that
\begin{equation}
\E R_0(E_1,E_2) = \E R_t(E_1,E_2)\left(1+\mathcal{O}(N^{-\delta_1}_2)\right).
\label{eq:R_zero_to_t}
\end{equation}
Next, from \eqref{eq:exptoresolvent} we have that
\begin{equation}
\E R_0(E_1-\ell_1,E_2-\ell_2) -\mathcal{O}(N_2^{-D})\le \mathbf{P}\left[\mathfrak{F}_1(0, x_1)\cap \mathfrak{F}_{2}(0, x_2)\right]\le \E R_0(E_1+\ell_1,E_2+\ell_2) +\mathcal{O}(N_2^{-D}).
\end{equation}
Combining   \eqref{eq:R_zero_to_t} for the shifted arguments,  \eqref{eq:R_shift} to remove the shifts
and \eqref{eq:R_factor} we arrive at
\begin{equation}
\mathbf{P}\left[\mathfrak{F}_1(0, x_1)\cap \mathfrak{F}_1(0, x_2)\right] = \mathbf{P}\left[\mathfrak{F}_1(t, x_1)\right]\mathbf{P}\left[\mathfrak{F}_{2}(t, x_2)\right]\left(1+\mathcal{O}(N^{-\delta_1}_2)\right).
\label{eq:P0_factor}
\end{equation}
\nc
To conclude the argument, we now need that 
\begin{equation} \label{eq:plugin}
\Prob \left[\mathfrak{F}_l(t, x_l)\right] = \Prob \left[\mathfrak{F}_l(0, x_l)\right] \, \big( 1 + \mathcal{O}(N_2^{-\delta_1})\big) \qquad l =1,2 \,. 
\end{equation}
This follows similarly to the proof of \eqref{eq:R_zero_to_t} with \eqref{eq:dif_R_aim} replaced by its one-point version,
see  \cite[Eq. (2.57)]{small_dev} for an analogous argument. \nc
By plugging \eqref{eq:plugin} into \eqref{eq:P0_factor}, we complete the proof of Proposition~\ref{prop:GFT}. \qed 
\vspace{2mm}

We are hence left with giving the proof of Proposition \ref{prop:GFTunderlying}. 

\begin{proof}[Proof of Proposition \ref{prop:R_dif}] The argument is similar to \cite[Proof of Theorem~2.4]{small_dev} and we will hence be very brief.  For simplicity, we focus on the real symmetric case, the computations in the complex Hermitian setting are analogous. Let $\partial_{ab}$ denote the derivative with respect to $x_{ab}$ for $a\le b$. Then, for any $D > 0$, by choosing $K = K(D)$ sufficiently large, we have
	\begin{equation}
		\partial_t \E R_t = \frac{1}{2}\E \sum_{a\le b} \left(-x_{ab}\partial_{ab} R_t + (1+\delta_{ab})\partial_{ab}^2 R_t\right) = -\frac{1}{2}\sum_{k=2}^K\sum_{a\le b} \frac{\kappa_{ab}^{(k+1)}(t)}{k!}\E\left[\partial_{ab}^{k+1} R_t\right] + \mathcal{O}(N_2^{-D}),
		\label{eq:R_dif}
	\end{equation}
	where the summation runs over $1\le a\le b\le N_2$. The normalization is chosen so that $\kappa_{ab}^{(k+1)}(t)$, which is the cumulant of $x_{ab}(t)$ of order $k+1$, is of order 1, and we gain smallness through differentiation with respect to $x_{ab}$. Recalling the definition of $R_t$ from \eqref{eq:def_R}, 
	and dropping the index $E_l$  from $\X_{E_l,l,t}$, \nc we compute $\partial_{ab}^{k+1} R_t$. 
	We find that it can be expressed as a sum of terms of the form
	\begin{equation}
		F^{(p)}(\X_{1,t})F^{(q)}(\X_{ 2,t})\prod_{r=1}^p \partial_{ab}^{\alpha_r}\X_{1,t}\prod_{s=1}^q \partial_{ab}^{\beta_s}\X_{2,t},
		\label{eq:R_term}
	\end{equation}
	where
	\begin{equation*}
	1 \le 	p+q\le k+1, \quad \text{with} \quad p,q\in\mathbf{N}_{0}\quad\text{and}\quad \alpha_r, \beta_s \ge 0 \quad \text{are such that} \quad \sum_{r=1}^p \alpha_r + \sum_{s=1}^q \beta_s=k+1.
	\end{equation*}
	
For estimating \eqref{eq:R_term}, we point out that there will be two sources of smallness: First, every derivative $\partial_{ab}$ will gain one factor of $N_l^{-1/2}$, simply by our normalization \eqref{eq:Hdef}. In this way, at order $k$, the term in \eqref{eq:R_term} naturally carries a small factor $N_l^{-(k+1)/2}$ since $\sum \alpha_r + \sum \beta_s = k+1$. Second, near the spectral edge, imaginary parts of resolvents are especially small as expressed by the single-resolvent isotropic local law for Wigner matrices \cite[Eq. (2.30)]{small_dev}:
	\begin{equation} \label{eq:locallaw}
		\max_{i,j \in [N_l]} \left| (G_t^{(N_l)})_{ij}(z) - \delta_{ij}m_{\mathrm{sc}}(z) \right| \prec N_l^{-1/3 + \varepsilon},  \quad |\Im m_{\rm sc}(z)| \lesssim N_l^{-1/3 + \varepsilon}, \quad l =1,2,
	\end{equation}
valid	uniformly in $t \ge 0$ and $z \in \C \setminus \R$ with $|\Re z - 2| \le N_l^{-2/3 + \varepsilon}$ and $N_l^{-2/3 - \varepsilon} \le |\Im z| \le  N_l^{-2/3 + \varepsilon}$. In particular, this means that $|G_{ij}|\prec 1$ and  $|\Im G_{ij}|\prec N_l^{-1/3+\varepsilon}$ in our regime. 
We now discuss how these two sources of smallness play together in estimating \eqref{eq:R_term}. 
	
Concerning the first source of smallness, we compute
	\begin{equation*}
		\partial_{ab}\X_{l,t} = \frac{1}{\sqrt{N_l}}\frac{2-\delta_{ab}}{2\pi}\Im \big(G_{ab}^{(N_l)}(y+\ii\eta_l)+ G_{ba}^{(N_l)}(y+\ii\eta_l) \big)\Big|_{E_l}^{E_L} =\frac{1}{\sqrt{N_l}} \frac{2-\delta_{ab}}{2\pi}\Delta_l\Im \big( G_{ab}^{(N_l)}+G_{ba}^{(N_l)}\big),
	\end{equation*}
	where, for fixed $E_l, E_L$ and $\eta_l$, we define the finite difference operator
	\begin{equation*}
	\Delta_l f:= f(E_L+\ii\eta_l)-f(E_l+\ii\eta_l) \qquad l=1,2
	\end{equation*}
	acting on any function $f=f(z)$.
 In particular, the factor $\partial_{ab}^{\alpha_r}\X_{1,t}$ (and similarly for $\partial_{ab}^{\beta_s}\X_{2,t}$) decomposes into the sum of terms of the form
	\begin{equation}
		\Delta_1 \Im \prod_{\gamma=1}^{\alpha_r} G_{a_\gamma b_\gamma}^{(N_1)}, \quad a_\gamma,b_\gamma\in\{a,b\}.
		\label{eq:X_dif}
	\end{equation}
	In order to get an upper bound on the resolvent product \eqref{eq:X_dif} we employ \eqref{eq:locallaw}. In this way, concerning
	the second source of smallness, we hence get, for $\alpha_r \ge 1$, that
	\begin{equation} \label{eq:singleFdiff}
		\left\vert\Delta_1 \Im \prod_{\gamma=1}^{\alpha_r} G_{a_\gamma b_\gamma}^{(N_1)}\right\vert \lesssim \sum_{\gamma=1}^{\alpha_r}\big|\Im G_{a_\gamma b_\gamma}^{(N_1)}\big| \prec N_1^{-1/3+\varepsilon}=:\Psi, 
	\end{equation}
	where we effectively gained one $\Im G$-smallness for each $\gamma \in [\alpha_r]$. 
	
	Thus, for the terms \eqref{eq:R_term} with $k\ge 3$, using that $p + q \ge 1$, we have
	\begin{equation}
		\left\vert\sum_{a\le b} F^{(p)}(\X_{1,t})F^{(q)}(\X_{2,t})\prod_{r=1}^p \partial_{ab}^{\alpha_r}\X_{1,t}\prod_{s=1}^q \partial_{ab}^{\beta_s}\X_{2,t}\right\vert\prec \left\vert F^{(p)}(\X_{1,t})F^{(q)}(\X_{2,t})\right\vert \Psi,
		\label{eq:term_bound}
	\end{equation}
	where we additionally used that $N_l^{-(k+1)/2}N_l^2\le 1$. 
	
For the third-order terms with $k=2$ in \eqref{eq:R_dif} the power counting is slightly more complicated, since the first source of smallness (yielding $N_l^{-(k+1)/2}$) does not balance the $N_l^2$-fold summation. Hence, one has to
 further exploit  the second source of smallness \eqref{eq:singleFdiff}. Note that the $\Psi$-factor from \eqref{eq:singleFdiff} can be gained whenever $p \ge 1$ (or $q \ge 1$, respectively). The most critical term is thus with $p+q = 1$, i.e.~either $p=0$ or $q=0$.  However, these cases with $p=0$ or $q=0$ (i.e.~when one of the $F$'s is not differentiated at all) have already been treated in \cite[Eq.~(2.50)]{small_dev} via \emph{isotropic resummation}, yielding the bound 
\begin{equation} \label{eq:hardterms}
		\left\vert\sum_{a\le b} F^{(p)}(\X_{1,t})F^{(q)}(\X_{2,t})\prod_{r=1}^p \partial_{ab}^{\alpha_r}\X_{1,t}\prod_{s=1}^q \partial_{ab}^{\beta_s}\X_{2,t}\right\vert\prec N^{-1/6 + 2 \varepsilon}  \left\vert F^{(p)}(\X_{1,t})F^{(q)}(\X_{2,t})\right\vert \,. 
\end{equation}
We may hence assume that $p,q\ge 1$, in which case we can gain two factors of $\Psi$ in \eqref{eq:term_bound}. The corresponding sum has an upper bound of order
	\begin{equation} \label{eq:easyterms}
		\begin{split}
			&\left\vert\sum_{a\le b} F^{(p)}(\X_{1,t})F^{(q)}(\X_{2,t})\prod_{r=1}^p \partial_{ab}^{\alpha_r}\X_{1,t}\prod_{s=1}^q \partial_{ab}^{\beta_s}\X_{2,t}\right\vert \\
		& \qquad \qquad 	\prec	N_2^{-3/2}N_2^2\Psi^2 \left\vert F^{(p)}(\X_{1,t})F^{(q)}(\X_{2,t})\right\vert  \lesssim N_2^{-1/6+2\varepsilon} \left\vert F^{(p)}(\X_{1,t})F^{(q)}(\X_{2,t})\right\vert .
		\end{split}
	\end{equation}
	Thus, combining \eqref{eq:hardterms} and \eqref{eq:easyterms} for $k=2$ with \eqref{eq:term_bound} for $k \ge 3$,  we infer
	\begin{equation} \label{eq:selfcons}
	|	\partial_t \E R_t |\lesssim  N_2^{-1/6+3\varepsilon}\sum_{p+q\le K+1} \E \left\vert F^{(p)}(\X_{1,t})F^{(q)}(\X_{2,t})\right\vert + \mathcal{O}(N_2^{-D}).
	\end{equation}

	To finish the argument, it is left to notice that, by construction of the function $F$, for any non-negative integers $p$ and $q$ we have
	\begin{equation}
		\left\vert\E \left[ F^{(p)}(\X_{1,t})F^{(q)}(\X_{2,t})\right]\right\vert \lesssim \E \left[ F\big(\chi_{E_1-2\ell_1}*\theta_{\eta_1}\big(H^{(N_1)}_t\big)\big) \, F\big(\chi_{E_2-2\ell_2}*\theta_{\eta_2}\big(H^{(N_2)}_t\big)\big)\right] + \mathcal{O}(N_2^{-D}),
		\label{eq:F_dif}
	\end{equation}
	where the implicit constant in the error term depends only on $p$ and $q$. The proof of \eqref{eq:F_dif} closely follows the proof of \cite[Eq.~(2.46)]{small_dev} and thus is omitted. By plugging \eqref{eq:F_dif} into \eqref{eq:selfcons}, we conclude the proof of Proposition~\ref{prop:GFTunderlying}. 
\end{proof}

\section{Proof of Corollary \ref{cor:limits}} \label{sec:corollary}

We  start with the proof of \eqref{closure-2/3}. From Theorem~\ref{thm:main}, we know that for any fixed
 $t\in (0, (2\beta)^{-2/3})$, 
\begin{align}
	\mathbf{P}\left[ \frac{\lambda_1^{(N+1)}}{(\log (N+1))^{2/3}} \leq t\leq \frac{\lambda_1^{(N)}}{(\log N)^{2/3}},  \quad \text{i.o.}\right]=1, \label{031701}
\end{align} 
where \emph{i.o.} is short for \emph{infinitely often}. 
This is because the sequence $\frac{\lambda_1^{(N)}}{(\log N)^{2/3}}$ needs to go between $(2\beta)^{-2/3}$ and $0$,  back and forth, infinitely many times, as with an extra $(\log N)^{-1/3}$ factor, $\liminf \frac{\lambda_1^{(N)}}{(\log N)^{2/3}}\to 0$.  By \eqref{031701}, we know that for infinitely many $N$, 
\begin{align}
	0\leq \Big|t-\frac{\lambda_1^{(N+1)}}{(\log (N+1))^{2/3}}\Big|&=t-\frac{\lambda_1^{(N+1)}}{(\log (N+1))^{2/3}}\leq \frac{\lambda_1^{(N)}}{(\log N)^{2/3}}- \frac{\lambda_1^{(N+1)}}{(\log (N+1))^{2/3}}\notag\\
	&=\frac{\lambda_1^{(N+1)}}{(\log (N+1))^{2/3}}\left(\frac{(\log (N+1))^{2/3}}{(\log N)^{2/3}}-1\right)+\frac{\lambda_1^{(N)}-\lambda_1^{(N+1)}}{(\log N)^{2/3}} \label{031713}
\end{align}
Clearly, we have, again from Theorem \ref{thm:main}, that
\begin{align}
	\frac{\lambda_1^{(N+1)}}{(\log (N+1))^{2/3}}\Big(\frac{(\log (N+1))^{2/3}}{(\log N)^{2/3}}-1\Big)\to 0, \quad \text{a.s.}. \label{031711}
\end{align}
Next, we show that for any given constant $\varepsilon>0$, 
\begin{align*}
	\limsup_{N\to \infty}\frac{\lambda_1^{(N)}-\lambda_1^{(N+1)}}{(\log N)^{2/3}}\leq \varepsilon, \qquad \text{a.s.}
\end{align*}
To this end, we recall \eqref{031710}. Then we have 
\begin{align*}
	&\mathbf{P} \left[\frac{\lambda_1^{(N)}-\lambda_1^{(N+1)}}{(\log N)^{2/3}}\geq   \varepsilon\right]\leq \mathbf{P} \left[ \frac{\lambda_1^{(N+1)}-\lambda_1^{(N)}}{(\log N)^{2/3}}\leq  -\varepsilon\right]\\
	&\qquad \leq \mathbf{P}\left[  \frac{1}{N^{1/3}} \Big(|\xi_1^{(N+1)}|^2-1\Big)\leq -\frac{\varepsilon}{2}(\log N)^{2/3}\right]+N^{-D}\leq 2N^{-D}
\end{align*}
for any $D>1$. Hence, by the Borel-Cantelli lemma, we know that almost surely, $(\lambda_1^{(N)}-\lambda_1^{(N+1)})/(\log N)^{2/3}$ goes beyond $\varepsilon$ for only finitely many $N$. This together with \eqref{031711} and \eqref{031713}, we know that there exist infinitely many $N$, such that 
\begin{align*}
	0\leq \Big|t-\frac{\lambda_1^{(N+1)}}{(\log (N+1))^{2/3}}\Big|\leq \varepsilon. 
\end{align*}
Due to the arbitrariness of $\varepsilon$, we know  that $t$ is a limit points of the sequence ${\lambda_1^{(N)}}/{(\log N)^{2/3}}$. This concludes the proof of \eqref{closure-2/3}. 

The proof of \eqref{closure-1/3} is analogous: From Theorem \ref{thm:main}, we know that for any $t\in ((-8/\beta)^{1/3}, \infty)$, we have 
\begin{align*}
	\mathbf{P}\left[ \frac{\lambda_1^{(N+1)}}{(\log (N+1))^{1/3}} \leq t\leq \frac{\lambda_1^{(N)}}{(\log N)^{1/3}},  \quad \text{i.o.}\right]=1
\end{align*}
Then the remaining proof is the same as that of \eqref{closure-2/3} with $(\log N)^{2/3}$ replaced by $(\log N)^{1/3}$. 
This concludes the proof of Corollary \ref{cor:limits}. \qed

\section{Extension to more general models}\label{sec:wignertype}

In Remark~\ref{rm:ext}  we presented a more general class of random matrix ensembles
(Wigner-type matrices with diagonal deformations) for which the law of fractional logarithm 
also holds. Here we explain how our proof for Wigner matrices can be modified to
cover this more general case without going into all technical details that are relatively standard.

The proof presented in Sections~\ref{sec:2}, \ref{app:decor} has three key ingredients: 
\begin{itemize}
\item[(i)]  Tail bound (Proposition~\ref{prop:tails});
\item[(ii)]  Decorrelation estimate (Proposition~\ref{prop:decor});
\item[(iii)] Martingale argument (Section~\ref{subsec:maxmart}, especially the lower bound \eqref{diff}).
\end{itemize}
Given these ingredients
the rest of the proof is insensitive to the actual ensemble.

As to (i), we proved the tail bound (Proposition~\ref{prop:tails}) by a GFT argument, essentially borrowed
from \cite{small_dev}, that connected the Wigner ensemble with a Gaussian ensemble
along a long time OU flow. This would not work directly for Wigner type matrices whose 
scDOS is not the semicircle law; especially not for internal edges
in case the scDOS is supported on several intervals.
 Alternatively,  one may use a  similar short time DBM approximation 
as in Section~\ref{subsec:DBM}. In essence, after time $t\ge N^{-1/3+\epsilon}$, the evolution
of the true eigenvalues $\widetilde\lambda^{(N)}_i(t)$ can be coupled 
to GUE/GOE eigenvalues $\mu_i^{(N)}(t)$ with  precision of order $N^{-2/3-\omega}$
as given in \eqref{eq:DBMminor}.  This precision is much finer than the $N^{-2/3}(\log N)^{2/3}$ 
or  $N^{-2/3}(\log N)^{1/3}$ scales on which the tail bounds live, therefore the 
tail bounds for $\mu$’s can be directly transferred to those for $\widetilde\lambda(t)$’s.
Finally we use a standard short time GFT argument to remove the small Gaussian component;
here the time is sufficiently short that the scDOS does not change. The proof also works 
 for internal endpoints of the support of $\rho$ in case it consists of 
several component separated by gaps of order at least $N^{-2/3+\epsilon}$. This distance
guarantees that the gaps do not close under the short time DBM and GFT flows.
We remark that the same short time DBM method has been  used to prove the
Tracy-Widom universality in many different situations; first 
 for random matrices with a small  Gaussian component in \cite{landon2017edge},
and also  without Gaussian component after combining DBM with GFT 
in \cite{Bou_extreme, Alt_band}. However, compared with these  universality proofs, here we need to 
keep very effective error terms both in the DBM and in the GFT arguments
since our estimate is in the moderate deviation regime.

Turning to (ii), the decorrelation estimate has two parts that are sensitive to the ensemble.
 On one hand we relied on a two-resolvent local law
to prove the overlap bound \eqref{eq:overlapbound}, on the other hand in Section~\ref{subsec:GFT} 
we performed a short time  GFT 
argument for $R_t$ defined in \eqref{eq:def_R}. This second step is unproblematic as it is completely analogous to
the GFT used in the tail bound in part  (i). The overlap bound is more involved. Multi-resolvent
local laws can be routinely proven by the {\it zigzag strategy} as it was done 
in \cite[Proposition~A.1]{minor} that extended \cite[Theorem~3.2]{eigenv_decorr}
assuming the necessary bounds on the deterministic approximation holds.
Thanks to the very simple variance structure, the deterministic approximation 
for resolvents of Wigner matrices $W$ or even of Wigner matrices with an arbitrary deterministic deformation matrix $D$
is relatively simple. For example the resolvent $(W+D-z)^{-1}$ can be well approximated by the solution $M=M(z)$ of 
the Dyson equation 
$$
    -M^{-1} = z-D +\langle M \rangle, \qquad (\Im M)(\Im z)>0 
 $$
i.e. $M$ is just the resolvent of $D$ at a shifted spectral parameter. Similarly, the
deterministic approximation of a product of two resolvents of the form
$(W-z_1)^{-1}(W+D-z_2)^{-1}$ can be explicitly written up and hence easily estimated in terms of $\langle D^2\rangle^{-1}$
(a special case of this was used in Section 4.2 of \cite{minor} which yielded
the overlap bound \eqref{eq:overlapbound}). 
The deterministic approximation $M$ for  the resolvent of Wigner type matrices $H$  is much more complicated even with
a diagonal deformation matrix (or even without any deformation) and it cannot be expressed as a simple resolvent;
it is  implicitly given via the unique solution of the corresponding nonlinear Dyson equation \cite{QVE}.
Nontrivial estimates containing $\langle D^2\rangle^{-1}$  on the 
deterministic approximation to the product of two resolvents $(H-z_1)^{-1}(H+D-z_2)^{-1}$ would be even more subtle
(compare it with  Lemma 4.7 of \cite{ETH_W-type} where a nontrivial  decay in $z_1-z_2$ was extracted but only for $D=0$;
 here we would need to extract the  nontrivial effect of $D$).  
 A much simpler alternative is to use the flatness condition $\Sigma_{ij} \ge c$ in  \eqref{eq:Hdefgen} to 
 approximate the Wigner type matrix  as $H\approx  H_0  + \frac{c}{2} W$, where $W$ is a standard Wigner matrix and $H_0$
 is an independent Wigner type matrix. Here the $\approx$ means that the
 first three moments match; such approximation is standard (see e.g \cite[Lemma 16.2]{erdHos2017dynamical}).  We then condition on $H_0$ and use the results from \cite[Proposition A.1]{minor} and \cite[Propositions 3.1 and 4.6]{eigenv_decorr} to estimate the product
 of resolvents of $H_0  + \frac{c}{2} W$ and $H_0 +D + \frac{c}{2} W$. Finally, we need to use another GFT
 to conclude about the product of resolvents of $H$ and $H+D$. Since the third moments match exactly,
 the powerful one by one Lindeberg  replacement strategy \cite{TaoVu11} and \cite[Chapter 16]{erdHos2017dynamical} is applicable and transfers the necessary bound of order $\langle D^2\rangle^{-1}$.

Finally, for the third ingredient (iii), the martingale argument used the specific unit variances
of the   Wigner ensemble  in the formula
$\mathbf{E} [ |\xi_{\alpha}^{(n)}|^2 \; | \mathcal{F}_{n-1}]=1$ to find that $X_{k,n}$ in \eqref{eq:martdef}
is a martingale. To preserve this property, 
we need to replace  $|\xi_{1}^{(\ell)}|^2 -1$  with $|\xi_{1}^{(\ell)}|^2 -
\mathbf{E} [ |\xi_{1}^{(\ell)}|^2 \; | \mathcal{F}_{\ell-1}]$, 
in \eqref{diff} and hence in the definition of $X_{k,n}$. The same change  in \eqref{mainid}--\eqref{eq:mainid3}
 and  simple algebraic properties
of the Stieltjes transform of the scDOS show that 
the net effect of this modification is to change the $2$ in the lhs.~of \eqref{mainid}--\eqref{eq:mainid2} and in \eqref{eq:mainid3} to the actual endpoint $E$ of the support of the 
scDOS for $H^{(N)}$. This is  exactly what is  needed to conclude that the martingale lower bound \eqref{diff}
still holds for the rescaled eigenvalue differences. Here, the correct shift and rescaling factor in 
the modified definition of $\lambda_1^{(N)}$ \eqref{eq:TWconv} takes into account that
$2$ is replaced with $E$ and an additional multiplicative factor accounts for
the constant in front of the asymptotic behavior $\rho(x)\approx c\sqrt{(E-x)_+}$ of the scDOS
near $E$. This argument holds also for internal endpoints $E$ of the support of $\rho$ if its components are
separated by gaps of order at least $N^{-2/3+\epsilon}$. 
Once \eqref{diff} is proven, the rest of Section~\ref{subsec:maxmart} is independent
of the actual ensemble.

\bibliographystyle{plain} 
\bibliography{refs}

\end{document}